\documentclass[10pt,amsfonts, epsfig]{amsart}
\usepackage{bbm}
\usepackage{amscd}

\usepackage{amsmath, amscd, amssymb}
\usepackage[frame,cmtip,arrow,matrix,line,graph,curve]{xy}
\usepackage{graphpap, color}
\usepackage[mathscr]{eucal}
\usepackage{mathrsfs}%,mathabx}
\usepackage{graphicx}
\usepackage{pstricks}
\usepackage{color}
\usepackage{cancel}
\usepackage{verbatim}

\numberwithin{equation}{section}

\newcommand{\CC}{\mathbb{C}}

\newcommand{\PP}{\mathbb{P}}
\newcommand{\QQ}{\mathbb{Q}}
\newcommand{\RR}{\mathbb{R}}
\newcommand{\ZZ}{\mathbb{Z}}

\newcommand{\bx}{\mathbf{x}}
\newcommand{\kk}{\Bbbk}

\newcommand{\cal}{\mathcal}

\def\cC{{\cal C}}
\def\cD{{\cal D}}
\def\cE{{\cal E}}

\def\cL{{\cal L}}
\def\cM{{\cal M}}
\def\cN{{\cal N}}
\def\cO{{\cal O}}

\def\cR{{\cal R}}

\def\cU{{\cal U}}

\def\mapright#1{\,\smash{\mathop{\lra}\limits^{#1}}\,}

\def\dual{^{\vee}}

\def\sta{^\ast}

\def\virt{^{\mathrm{vir}}}
\def\upmo{^{-1}}
\def\sta{^{\ast}}

\def\sta{^*}

\def\lra{\longrightarrow}

\def\lsta{_{\ast}}

\newcommand{\si}{\sigma}

\def\begeq{\begin{equation}}
\def\endeq{\end{equation}}
\def\and{\quad{\rm and}\quad}
\def\bl{\bigl(}
\def\br{\bigr)}
\def\defeq{:=}
\def\mh{\!:\!}
\def\sub{\subset}
\def\Ao{{{\mathbb A}^{\!1}}}

\def\Po{{\mathbb P^1}}
\def\and{\quad\text{and}\quad}
\def\mapright#1{\,\smash{\mathop{\lra}\limits^{#1}}\,}

\def\lalp{_\alpha}

\def\reg{_{\text{reg}}}

\DeclareMathOperator{\pr}{pr} 
 
 \DeclareMathOperator{\ext}{\cE
\it{xt}}

 \DeclareMathOperator{\rank}{rank}
\DeclareMathOperator{\spec}{Spec}

\newtheorem{prop}{Proposition}[section]
\newtheorem{theo}[prop]{Theorem}
\newtheorem{lemm}[prop]{Lemma}
\newtheorem{coro}[prop]{Corollary}
\newtheorem{rema}[prop]{Remark}
\newtheorem{exam}[prop]{Example}
\newtheorem{defi}[prop]{Definition}

\newtheorem{defiprop}[prop]{Definition-Proposition}

\DeclareMathOperator{\coker}{coker}

\def\mgn{\overline \cM_{g,n}}

\def\Ob{\cO b}

\def\loc{_{\mathrm{loc}}}

\def\bul{^\bullet}

\def\ev{\text{ev}}

\def\vsp{\vskip5pt}
\def\beq{\begin{equation}}
\def\eeq{\end{equation}}
\def\bul{^{\bullet}}

\def\bc{{\mathbf c}}

\def\bM{{\mathbf M}}
\def\be{{\mathbf e}}
\def\ti{\tilde}

\def\Msi{{\cM(\si)}}
\def\primesta{^{\prime\ast}}

\def\ti{\tilde}
\let\sO=\cO

\def\lstar{_{\ast-r}}
\def\bE{{\mathbf E}}

\title{Localizing virtual cycles by cosections} %, and applications}
\date{}

\author{Young-Hoon Kiem}
\address{Department of Mathematics and Research Institute
of Mathematics, Seoul National University, Seoul 151-747, Korea}
\email{kiem@math.snu.ac.kr}

\author{Jun Li}
\address{Department of Mathematics, Stanford University, Stanford,
USA} \email{jli@math.stanford.edu}

\thanks{Young-Hoon Kiem was partially supported by KOSEF grant R01-2007-000-20064-0;
Jun Li was partially supported by the NSF grant NSF-0601002. }

\begin{document}

\def\DM{Deligne-Mumford }
\begin{abstract} We show that a cosection of the obstruction sheaf of a
perfect obstruction theory localizes the virtual cycle to the non-surjective locus of the
cosection. We construct localized Gysin map
and localized virtual cycles.  Various applications of this construction
are discussed.
\end{abstract}
\maketitle %\tableofcontents

\section{Introduction}

Invariants defined by virtual cycles of moduli spaces have played 
important roles in research in algebraic geometry. Invariants of this
kind include the much studied Gromov-Witten (in short GW) invariants
of varieties, and the recently introduced Donaldson-Thomas (in short
DT) invariants of Calabi-Yau three-folds.

One of the main challenges in studying such invariants is to develop
techniques for investigating virtual cycles. In this paper, we will
present a new technique, which we call the localization by cosection
of obstruction sheaf (Theorem \ref{main}). This is achieved after
constructing a localized Gysin map (Proposition
\ref{loc-Gysin-intro}).

\begin{theo}[Localization by cosection]\label{main}
Let $\cM$ be a Deligne-Mumford stack endowed with a perfect
obstruction theory. Suppose the obstruction sheaf $\Ob_\cM$ admits a
surjective homomorphism $\si:\Ob_\cM|_U  \to \cO_U$
over an open $U\sub \cM$.
Let $\cM(\sigma)=\cM\setminus U$.
Then $(\cM,\sigma)$ has a localized virtual cycle
$$[\cM]\virt\loc\in A\lsta \cM(\sigma).
$$
This cycle enjoys the usual properties of the virtual cycles; it relates to
the usual virtual cycle $[\cM]\virt$ via 
$[\cM]\virt = \iota\lsta[\cM]\virt\loc \in A\lsta  \cM$, where
$\iota: \cM(\sigma)\to \cM$ is the inclusion.
\end{theo}

This work originated from our effort to understand Lee-Parker's
discovery that GW-invariants of surfaces with holomorphic two-forms
localize along the loci of stable maps to canonical divisors
\cite{Lee-Parker}. We show that a holomorphic two-form induces a
cosection of the obstruction sheaf of the moduli space; applying
localization by cosection we recover and generalize Lee-Parker's
theorem as follows.

Let $X$ be a smooth quasi-projective variety with a holomorphic
two-form $\theta\in H^0(\Omega^2_X)$. Let $\mgn(X,\beta)$ denote the
moduli stack of $n$-pointed stable maps of genus $g$ to $X$ with
homology class $\beta$. The two-form $\theta$ on $X$ induces a
cosection $\sigma$ of the obstruction sheaf of $\mgn(X,\beta)$; the
degeneracy loci $\cM(\sigma)$ consist of stable maps $[u:C\to X]$
satisfying $\theta(u\lsta TC)=0$ (called $\theta$-null stable maps)
where $\theta$ is viewed as an antisymmetric homomorphism $T_X\to
\Omega_X$. Applying the localization by cosection, we obtain

\begin{theo} For a pair $(X,\theta)$ of a smooth quasi-projective variety
and a holomorphic two-form, the virtual fundamental class of
$\mgn(X,\beta)$ vanishes unless $\beta$ is represented by a
$\theta$-null stable map.
\end{theo}

The localization by cosection has other applications in the study of
GW-invariants and of DT-invariants. One is the work on quantum
cohomology of Hilbert scheme of points by W.-P. Li and the second
named author. Using meromorphic two-forms of algebraic surfaces,
they determined the two-point extreme quantum cohomology of the
Hilbert scheme of points of any algebraic surface \cite{LL}.

A special but important case is when the cosection is regular and surjective
everywhere. In this case, a reduced virtual cycle can be defined using
localization by cosection. The notion of reduced virtual fundamental class was first introduced by 
Okounkov-Pandharipande in \cite[\S3.4.3]{BGP} for GW-invariants of symplectic varieties. In 
studying curve counting \cite{MPT}, 
Maulik-Pandharipande-Thomas use the localization of virtual cycle by cosection 
to define the reduced Gromov-Witten and Pandharipande-Thomas invariants of a class
of Calabi-Yau three-folds; in the appendix of the same paper, they provide an in-depth discussion
on the relationships between reduced class construction and usual deformation-obstruction theories. 

Another application of reduced virtual cycle appears in a recent paper \cite{KL3}:
the authors use the localization
by cosection to prove a $\CC\sta$-wall crossing formula of
DT-invariants by producing a reduced virtual fundamental class of master space.

The proof of localization by cosection (Theorem \ref{main}) consists
of two parts. In one part, we prove that the intrinsic normal cone
of the perfect obstruction theory of $\cM$ lies in the kernel cone
stack of the cosection $\sigma$. In particular, if there is a
surjective homomorphism $\si: \Ob_\cM\to\cL$ to a locally free sheaf
$\cL$, the intrinsic normal cone lies in a smaller cone stack and we
can define a \emph{reduced virtual fundamental class}.

The other part of the proof is the construction of the following
localized Gysin map.
\begin{prop}[Localized Gysin Map]\label{loc-Gysin-intro}
Let $E$ be a rank $r$ vector bundle on a \DM stack $\cM$, and let
$\sigma$ be a meromorphic surjective cosection of $E$, meaning that
there is an open $U\sub \cM$ so that $\sigma$ is a surjective
homomorphism $\sigma: E|_U\to \cO_U$. Let $\cM(\si)=\cM\setminus U$
and let $E(\si)=E|_{\cM(\si)}\cup \ker\{\si: E|_U\to\cO_U\}$. The
Basic Construction stated in Section 2 defines a homomorphism
$$s_{E,\si}^!: A\lsta E(\si)\to A\lstar\Msi,
$$
which we call the localized Gysin map. Furthermore, if we denote the
inclusions by $\iota: \Msi\to \cM$, $\ti \iota: E(\si)\to E$, and
let $s_E: A\lsta E\to A_{\ast-r}\cM$ be the usual Gysin map of
intersecting with the zero section of $E$, then we have
$$\iota\lsta\circ s_{E,\si}^!=s_E^!\circ\ti\iota\lsta: A\lsta
E(\si)\to A\lstar \cM.$$
\end{prop}
By applying the localized Gysin map to the intrinsic normal cone in
$E(\sigma)$, we obtain the localized virtual cycle $[\cM]\virt\loc$
in Theorem \ref{main}.
\vsp

%The application of localization by cosection to the GW-invariants of
%varieties with holomorphic two-forms follows after constructing a
%cosection of the obstruction sheaf using the holomorphic two-form.

The paper is organized as follows. In Sections 2 and 3, we
construct the localized Gysin map. In Section 4, we show that the
intrinsic normal cone lies in the kernel cone stack. In Section 5,
we define the localized virtual cycles and prove its deformation
invariance. An application of this localization technique to
GW-invariants of varieties with holomorphic two-forms is presented
in Section 6. In Section 7, we discuss other possible
applications. In the Appendix, we give an analytic definition of
the localized Gysin map and prove its equivalence with the
algebraic definition.

\vsp

\noindent {\bf Addendum}. The current version is the replacement of
the first half of \cite{Kiem-Li}. Our prior treatment of
localization by cosection used a topological definition of localized
Gysin map \cite{Kiem-Li}, which limits its application in algebraic
geometry.  In this paper, we provide an algebraic construction of
localized Gysin map, including the Chow groups of the total space of
a cone-stack over a \DM stack. This makes it possible to directly
apply other developed techniques on virtual cycles to localized
virtual cycles.

In the sequel of this paper, we will prove a degeneration formula of
localized GW-invariants, and include its application worked out in
\cite{Kiem-Li} in proving Maulik-Pandharipahnde's conjecture
(\cite{MaulikPand}) on degree two GW-invariants of surfaces. \vsp

\noindent {\bf Notation}: In this paper, all schemes and stacks are
defined over the complex number field $\CC$. We will use $Z\lsta X$
(resp. $A\lsta X$; resp. $W\lsta X$) to denote the group of
algebraic cycles (resp. group of algebraic cycles modulo rational
equivalence; resp. group of rational equivalences) with
$\QQ$-coefficients.

Since we will be working with locally free sheaves 
and cycles in the total spaces of the vector
bundles associated to the locally free sheaves, to streamline the
notation, we will use the same symbol to denote a locally free sheaf
as well as its associated vector bundle. Thus, given a vector bundle
(locally free sheaf) $E$, by $Z\lsta E$ we mean the group of cycles
of the total space of $E$, and by $E\to \Ob_\cM$ we mean a sheaf
homomorphism $\cO_\cM(E)\to \Ob_\cM$.

Given a subvariety $T\sub E$, we
denote by $[T]\in Z\lsta E$ its associated cycle, and denote by $[T]\in A\lsta E$ its rational equivalence class
in $A\lsta E$.

\smallskip
\noindent{\bf Acknowledgment}: The first named author is grateful to
the Stanford Mathematics department for support and hospitality
while he was visiting during the academic year 2005/2006.
%The second author
%thanks D. Maulik for sharing with him his computation for an example
%that is crucial for the second part of the paper; he also thanks E.
%Ionel for stimulating discussions.
We thank J. Lee and T. Parker for
stimulating questions and for pointing out several oversights in our
previous draft.

\def\cS{{\mathcal S}}

\section{Localized Gysin maps}
Let $\pi: E\to \cM$ be a rank $r$ vector bundle over a \DM (DM for
short) stack $\cM$. The usual Gysin map $s_E^!: A_d E\to A_{d-r}
\cM$ is defined by ``intersecting'' cycles in $E$ with the zero
section $s_E$ of $E$.

In this section, we suppose that $E$ has a surjective meromorphic
cosection $\sigma$.

\begin{defi}\label{def-1}
A surjective meromorphic cosection is a surjective homomorphism
$\sigma: E|_U\to\cO_U$ for an open substack $U\sub \cM$.
\end{defi}

\begin{rema}
For the purpose of defining localized Gysin map, the degeneracy
locus of a meromorphic cosection $\si: E-\!\to \cO_\cM$ includes the
loci where $\si$ is undefined and where $\si$ is not surjective. In
the definition above, the open $U$ is the locus where $\si$ is
defined and surjective.
\end{rema}

For such a $\sigma$, we let $\cM(\sigma)=\cM\setminus U$, let $G=\ker\{\sigma: E|_U\to\cO_U\}$ and let
$E(\sigma)=E|_{\cM(\sigma)}\cup G$, which is closed in $E$.
The goal of this section is to define a localized Gysin map
$$s^!_{E,\sigma}: A_d E(\sigma)\to A_{d-r} \cM(\sigma)
$$
that has the usual properties of the Gysin map and coincides with $s_E^!$ when composed with the
tautological $A_{d-r}\Msi\to A_{d-r}\cM$.
\vskip3pt

\begin{defi}\label{def-reg}
Let $\rho: X\to\cM$ be a morphism from a variety $X$ to $\cM$ such that $\rho(X)\cap U\ne \emptyset$.
We call $\rho$ a $\si$-regularizing morphism if $\rho$ is proper, and $\rho\sta(\si)$ extends
to a surjective homomorphism
$$\ti\si: \ti E\defeq \rho\sta E\to \cO_X(D)
$$
for a Cartier divisor $D\sub X$. We adapt the convention that $\ti\rho: \ti E\to E$
is the projection; $\ti G\defeq \ker\{\ti\si\}\sub \ti E$; 
$|D|\sub X$ is the support of $D$, and $\rho(\si): |D|\to
\cM(\si)$ is the $\rho$ restricted to $|D|$.
\end{defi}

\noindent {\bf Basic Construction}: {\sl Let $[B] \in Z_d E(\si)$ be
a cycle represented by a closed integral substack $B\sub E(\si)$.
In case $B\sub E|_{\Msi}$, we define
$s_{E,\si}^!([B])=s_{E|_\Msi}^!([B])\in A_{d-r}\Msi$. Otherwise, we
pick a variety $X$ and a $\si$-regularizing $\rho: X\to \cM$ such
that there is a closed integral $\ti B\sub \ti G$ so that
$\ti\rho\lsta([\ti B])=k\cdot [B]\in Z_d E$ for some $k\in\ZZ$. We
define \beq\label{Gysin-M} s_{E,\si}^!([B])_{\rho,\ti B}=k\upmo\cdot
\rho(\si)\lsta([D]\cdot s^!_{\ti G}( [\ti B]))\in A_{d-r} \cM(\si).
\eeq Here $[D]\cdot : A\lsta X\to A_{\ast-1}|D|$ is the intersection
with the divisor $D$.}

\begin{lemm}\label{exist-G}
Let the notation be as in the basic construction. Then for each
closed integral $B\sub E(\si)$ not contained in $E|_{\Msi}$, we can
find a pair $(\rho,\ti B)$ so that $s_{E,\si}^!([B])_{\rho,\ti B}$
is defined. Furthermore the resulting cycle class
$s_{E,\si}^!([B])_{\rho,\ti B}\in A_{d-r}\Msi$ is independent of the
choice of $(\rho,\ti B)$.
\end{lemm}

\begin{proof}
We let $B\sub E$ be as in the statement of the Lemma;
we let $B_0=\pi(B)\sub\cM$, where $\pi: E\to \cM$ is the projection. To pick
$\rho$, we then pick a normal variety
$X$ and a proper and generically finite morphism $\rho: X\to B_0$.
Since $\cM$ is a DM-stack, such $\rho$ exists. Using $\rho$, we
pull back  $\ti E=\rho\sta E$, and $\rho\sta\si: \rho\sta
E|_{\rho\upmo U}\to \cO_{\rho\upmo U}$. Since $B\times_\cM U\ne\emptyset$,
$\rho\upmo U\ne \emptyset$.

Next, possibly after replacing $X$ by a blow-up of $X$, we can
assume that $\rho\sta\si$ extends to a surjective homomorphism
$\ti\si: \ti E\to \cO_X(D)$ for a Cartier divisor $D\sub X$.
Thus $\rho$ is a $\si$-regularizing morphism.
We let $\ti\rho: \ti E\to E$ and
$\ti G\sub \ti E$ be as defined in Definition \ref{def-reg}.

Since $\rho: X\to B_0$ is generically finite, there is an open
$O\sub B_0$ so that $\rho|_{\rho\upmo O}: \rho\upmo O\to O$ is
flat and finite. We let $\ti B$ be an irreducible component of
$\overline{\ti\rho\upmo(B\cap\pi\upmo(O))}$. Since $B$ is integral
and $\ti\rho$ is proper, $\ti\rho( \ti B)=B$, thus
$\ti\rho\lsta([\ti B])=k[B]$ for an integer $k$. This shows that
$s_{E,\si}^!([B])_{\rho,\ti B}$ is defined according to the Basic Construction.

We next check that $s_{E,\si}^!([B])_{\rho,\ti B}$ is independent of the choice of $(\rho,\ti B)$.
Let $\rho': \ti X'\to \cM$, $\ti E'\defeq \rho\primesta E$ and $\ti B'\sub \ti E'$ be another
choice that fulfills the requirement of the Basic Construction,
thus giving rise to the class $s_{E,\si}^!([B])_{\rho',\ti B'}\in A_{d-r}\Msi$.
We show that $s_{E,\si}^!([B])_{\rho,\ti B}=s_{E,\si}^!([B])_{\rho',\ti B'}$.

We form $Y=X\times_\cM X'$; denote by $q: Y\to X$ and $q': Y\to X'$ the projections, and by
$p: Y\to\cM$ the composite $\rho\circ q=\rho'\circ q'$. Since $\rho: X\to \cM$ is generically
finite and $\ti B\to X$ is dominant, $\ti B\times_E \ti B'$ contains a pure dimension $d$
irreducible component $\bar B$ that surjects onto $\ti B$ and $\ti B'$ via the tautological projections
$\ti q: \bar E\defeq p\sta E\to \ti E$ and
$\ti q': \bar E\to \ti E'$, respectively.

We claim that there is an isomorphism $\cO_Y(q\sta D)\cong \cO_Y(q\primesta D')$ so
that
\beq\label{equal-1}
q\sta \ti\si=q\primesta \ti \si': \bar E=q\sta \ti E=q\primesta\ti E'\lra \cO_Y(q\sta D)\cong \cO_Y(q\primesta D').
\eeq
Indeed, since $q\sta \ti\si|_{p\upmo U}=q\primesta \ti \si'|_{p\upmo U}=p\sta\sigma|_{p\upmo U}$,
the kernel sheaves $\ker\{ q\sta \ti\si\}|_{p\upmo U}=\ker\{q\primesta \ti \si'\}|_{p\upmo U}$.
Since both are subsheaves of the locally free sheaf $\bar E$, the two kernels are identical.
Therefore, since both $q\sta \ti\si$ and $q\primesta \ti \si'$ are surjective, we obtain \eqref{equal-1}.
Furthermore, since when restricted to $p\upmo U$ the last isomorphism in \eqref{equal-1} is the identity, 
the arrow in \eqref{equal-1}
sends
$1\in\Gamma(\cO_{Y}(q\sta D)|_{p\upmo U})$
to $1\in \Gamma( \cO_{Y}(q\primesta D')|_{p\upmo U})$. This conclude that
$q\sta D=q\primesta D'$ as divisors. We denote this divisor by $\bar D$.

We let $\bar\si=q\sta \ti\si=q\primesta \ti \si'$, and let $\bar G=\ker\{\bar\si\}$.
By \eqref{equal-1}, $\bar G$ is the pull-back
of $\ti G$ via $q\sta$ and of $\ti G'=\ker\{\ti\si'\}$ via $q\primesta$.
Back to $\bar B$, since $B\not\sub E|_\Msi$, $B\sub E(\si)$ and $\ti p(\bar B)=B$, we have $\bar B\sub \bar G$.
%Like $\ti\rho$, we denote by $\ti q$ (resp. $\ti q'$) the tautological projection $\bar E\to\ti E$ (resp. $\bar E\to \ti E'$).
We let $k_i$ be the integers so that
$$\ti q\lsta ([\bar B])=k_1 [\ti B],\quad \ti q'\lsta ([\bar B])=k_2[\ti B'],\quad
\ti\rho\lsta([\ti B])=k_3[B],\quad \ti\rho'\lsta([\ti B'])=k_4[B].
$$
Then since $\ti \rho\circ\ti q=\ti \rho'\circ\ti q'$, we
have $k_1k_3=k_2k_4$.

We let $p(\si): |\bar D|\to \Msi$, $q(\si): |\bar D|\to |D|$ and $\rho(\si): |D|\to\Msi$
be the restrictions of $p$, $q$ and $\rho$ to $|\bar D|$ and $|D|$, respectively.
Since the projections $\ti q$ and $\ti\rho$ are proper and because $\ti q\upmo \ti G=\bar G$,
applying the projection formula (and using $\bar D=q\sta D$), we obtain
$$\ s_{E,\si}^!([B])_{p,\bar B}=(k_1k_3)\upmo\cdot p(\si)\lsta([\bar D]\cdot  s_{\bar G}^!([\bar B]))
=(k_1k_3)\upmo\cdot \rho(\si)\lsta([D]\cdot s_{ \ti G}^!([\ti q\lsta \bar B]))
%=\rho(\si)\lsta\bl q(\si)\lsta (q\sta\ti D\cdot s_{\ti q\upmo(\ti G)}^!([\bar B]))\br
$$
$$\hskip-1.4cm
=k_3\upmo\cdot \rho(\si)\lsta( [D]\cdot s_{ \ti G}^!([\ti B]))=s_{E,\si}^!([B])_{\rho,\ti B}.
$$
Similarly, we have $s_{E,\si}^!([B])_{p,\bar B}=s_{E,\si}^!([B])_{\rho',\ti B'}$.
This proves the Lemma.
\end{proof}

The Lemma shows that the Basic Construction defines a homomorphism
\beq\label{Gysin-2} s_{E,\si}^!: Z_d E(\si)\lra A_{d-r} \cM(\si).
\eeq We next check that this homomorphism descends to a homomorphism
$A_d E(\si)\to A_{d-r}\cM(\si)$. We need a simple Lemma. Let $0\to
F_1\to F_2\to F_3\to 0$ be an exact sequence of vector bundles
(locally free sheaves) on a variety $Y$. Let $r_i=\rank F_i$ and
$p_i: F_i\to Y$ be the projections. Let $\iota: F_1\to F_2$ be the
inclusion.

\begin{lemm}\label{sub}
For any cycle $\alpha\in A_d F_1$, we have
$$s_{F_2}^!(\iota\lsta \alpha)=c_{r_3}(F_3)\cap s^!_{F_1}(\alpha)= s^!_{F_1}(c_{r_3}(F_3)\cap\alpha)
$$
\end{lemm}

\begin{proof}
By \cite[Thm 3.3]{Fulton}, there is a cycle $B\in Z_{d-r_1} Y$ so that $\alpha=p_1\sta B\in A\lsta F_1$.
Since $p_1\sta B$ is the total space of $F_1$ restricted to $B$, its Segre class is $s(p_1\sta B)=c(F_1)\upmo\cap
B$, where $c(F_1)$ is the total Chern class of $F_1$. By the formula expressing Gysin map in terms of
the Segre class \cite[Prop. 6.1]{Fulton},
$$s_{F_2}^!(\iota\lsta \alpha)=s_{F_2}^!(p_1\sta B)=[c(F_2)\cap(c(F_1)\upmo \cap B)]_{d-r_2}
= [c(F_3)\cap B]_{d-r_2}=c_{r_3}(F_3)\cap B. %\in A_{d-r_2} Y.
$$
On the other hand,
$$s^!_{F_1}(c_{r_3}(F_3)\cap\alpha)=
s^!_{F_1}(c_{r_3}(F_3)\cap [p_1\sta B])=
s^!_{F_1}(p_1\sta (c_{r_3}(F_3)\cap[B]))=c_{r_3}(F_3)\cap[B].%s_{F_2}^!(\alpha)= [c(F_3)\cap B]_{d-r_2}=c_{r_3}(F_3)\cap B \in A_{d-r_2} Y.
$$
This proves the Lemma.
\end{proof}

The same method proves the following Lemma that will be useful later.
Let $\rho': X'\to\cM$ be a $\si$-regularizing morphism; let
%proper morphism from an integral $X'$ to $\cM$ so that
%$\rho'(X')\cap U\ne \emptyset$ and that
%$\rho\primesta\sigma$ extends to a surjective
$\ti \si': \ti E'= \rho\primesta E\to \cO_{X'}(D')$,
$\ti G'=\ker\{\ti\si'\}$,
%Cartier divisor $D'$.  Let $\ti\rho'$ be the projection $\ti E'\to E$,  $\ti G'$ be the kernel of $\ti\si'$,
and $\rho'(\si): |D'|\to \Msi$ be the tautological projection, as in the proof of Lemma \ref{exist-G}.

\begin{lemm} \label{extra-2}
Let the notation be as before.
Suppose $[\ti B']\in Z_d \ti G'$
such that $\ti\rho'\lsta [\ti B']=0\in Z_d E$. Then $\rho'(\si)\lsta([D']\cdot s_{\ti G'}^!([\ti B'])=0\in A_{d-r}\Msi$.
\end{lemm}

\begin{proof}
%We follow the notation introduced in proving the independence part of the previous Lemma.
We let $B=\ti\rho'(\ti B')$. Since $\ti\rho'$ is proper, $B$ is closed and integral in $E$; since
$\ti\rho'\lsta([\ti B'])=0$ in $Z_d E$, $\dim B< d$. We let $B_0=\pi (B)\sub \cM$.
%In case
%$B_0\sub \Msi$, then $\ti B'\sub \ti E'|_{X'(\si)}$. Using Lemma \ref{sub} and that $\ti\rho'$ is proper,
%$$\rho'(\si)\lsta(D'\cdot s_{\ti G'}^!([\ti B'])=\rho'(\si)\lsta \bl s_{\ti E'|_{|D'|}}^!([\ti B'])\br =
%s^!_{E|_{\Msi}}(\ti\rho'\lsta[\ti B'])=0.
%$$
%The Lemma holds.
%In case $B_0\not\sub\Msi$,
We then pick $\rho: X\to\cM$ a $\si$-regularizing morphism so that
$\rho(X)=B_0$ and $\rho: X\to B_0$ is generically finite.

We form $Y=X\times_\cM X'$ and let $p: Y\to\cM$ be the induced
morphism. As in the proof of Lemma \ref{exist-G} (and using the
notations developed in the proof), we have $\bar E=p\sta E\to Y$,
$\ti q': \bar E\to \ti E'$ and an integral $\bar B\sub \bar E$ so
that $\ti q'(\bar B)=\ti B'$. Let $\ti q: \bar E\to \ti E$ and let
$\ti B=\ti q(\bar B)$. Since $\ti B$ is integral, since $\ti\rho(\ti
B)=B$, since $B\to B_0$ is dominant, and since $\rho: X\to B_0$ is
dominant and generically finite, $\dim \ti B=\dim \ti\rho(\ti
B)=\dim B_0<d$. Thus $\ti q\lsta([\bar B])=0\in Z_d \ti E$.
Therefore,
$$\rho'(\si)\lsta([D']\cdot s_{\ti G'}^!([\ti B'])=p(\si)\lsta([\bar D]\cdot s_{\bar G}^!([\bar B]))
=\rho(\si)\lsta([D]\cdot s_{\ti G}^!([\ti q\lsta \bar B]))=0.
$$
This proves the Lemma.
\end{proof}

We next show that the map $s_{E,\si}^!$ preserves rational equivalence.
In this paper, we adopt the the convention that a rational equivalence $R\in W_d E$ is a
sum $R=\sum_{\alpha\in\Lambda}  n\lalp[S\lalp,\phi\lalp]$,
where $S\lalp\sub E$ are $(d+1)$-dimensional closed integral and $\phi\lalp\in\kk(S\lalp)\sta$. We also use
$\partial_0 [S,\phi]=(\phi=0)\cdot [S]$ and $\partial_\infty [S,\phi]=(\phi=\infty)\cdot [S]$.

\begin{lemm}\label{Gysin-1}
For $R\in W_d E(\si)$,
$s_{E,\si}^!(\partial_0 R)=s^!_{E,\si}(\partial_\infty R)$.
\end{lemm}

\begin{proof}
We only need to prove the Lemma for $R=[S,\phi]\in W_d E(\si)$, where $S\sub E(\si)$ is closed and
integral. In case $S\sub E|_\Msi$, then
both $\partial_0 R$ and $\partial_\infty R\in Z_d E|_\Msi$, and the required identity holds because
the usual Gysin map $s_{E|_\Msi}^!: Z_d (E|_\Msi)\to A_{d-r}\Msi$ preserves the rational equivalence.

We now suppose $S\not\sub E|_\Msi$. We let $S_0=\pi(S)\sub \cM$.
We pick a proper and generically finite $\si$-regularizing $\rho: X\to \cM$
so that $\rho(X)=S_0$. We let $\ti \si: \ti E\defeq \rho\sta E\to \cO_X(D)$ be the surjective
homomorphism extending $\rho\sta\si$. By Lemma \ref{exist-G},
such $\rho$ exists. %Furthermore,
Because $\rho$ is generically finite, we can find a closed integral
$d+1$ dimensional $\ti S\sub \ti E$ so that with $\ti\rho: \ti E\to E$ the projection, $\ti\rho( \ti S)=S$.

We let $\ti G=\ker\{ \ti \sigma: \ti E\to \cO_X(D)\}$.
Because $S\sub E(\si)$ and $S\not\sub E|_\Msi$, $\ti S\sub \ti G$.
%Hence $[\ti S,\ti \phi]\in W_d \ti E(\ti\si)$.
We let $\ti\phi$ be the pull-back $\ti\rho\sta\phi$, and let $\ti R=[\ti S,\ti\phi]$. Thus, $\ti\rho\lsta\ti R\in W_d E(\si)$ and
$\ti\rho\lsta\ti R=k\cdot R$, for an integer $k$;
since $\ti\rho$ is proper,
\beq\label{proj}
\ti\rho\lsta(\partial_0 \ti R)=k\cdot \partial_0 R
\and \ti\rho\lsta(\partial_\infty \ti R)=k\cdot \partial_\infty R.
\eeq

We now decompose the cycle $\partial_0 R=B_1+B_2$ so that
each summand $[T]\in B_1$ (resp. $[T]\in B_2$) satisfies $T\sub E|_{\cM(\si)}$
(resp. $T\not\sub E|_{\cM(\si)}$).
We also decompose $\partial_\infty R=C_1+C_2$ according to the same rule with
$B_i$ replaced by $C_i$.

For $\partial_0\ti R$, we decompose it into the sum of three parts
$\partial_0\ti R=\ti B_0+\ti B_1+\ti B_2$ so that $[T]\in \ti B_0$ (resp. $[T]\in \ti B_1$; resp. $[T]\in \ti B_2$)
satisfying $\ti\rho\lsta([T])=0$ (resp. $T\sub \ti E|_{|D|}$; resp. $T\not\sub \ti E|_{|D|}$).
By moving all summands $[T]$ with $\ti\rho\lsta[T]=0$ to $\ti B_0$, no $[T]$ appears simultaneously in two of the
three factors
$\ti B_0$, $\ti B_1$ and $\ti B_2$.
We decompose $\partial_\infty\ti R=\ti C_0+\ti C_1+\ti C_2$ according to the same rule with
$\ti B_i$ replaced by $\ti C_i$.
By \eqref{proj}, and using the property of the decompositions, we have
$\ti\rho\lsta(\ti B_i)=k\cdot B_i$ and $\ti\rho\lsta (\ti C_1)=k\cdot C_i$ for $i=1,2$.

Applying the definition of $s_{E,\si}^!$, to $\partial_0 R$ (resp. $\partial_\infty R$):
$$s_{E,\si}^!(\partial_0 R)=s_{E|_\Msi}^!(B_1)+ k^{-1} \rho(\si)\lsta ([D]\cdot s^!_{\ti G}(\ti B_2)),
$$
(resp. same formula with $\partial_0 R$ replaced by $\partial_\infty R$), we obtain identities in $A_{d-r}\Msi$:
$$s_{E,\si}^!(\partial_0 R)-s_{E,\si}^!(\partial_\infty R)=s_{E|_\Msi}^!(B_1-C_1)+
k\upmo\cdot \rho(\si)\lsta ([D]\cdot s^!_{\ti G}( \ti B_2-\ti C_2)).
$$
%$$\qquad\qquad\qquad\qquad= s_{E|_\Msi}^!(B_1-C_1)+k\upmo \rho(\si)\lsta (\ti D\cdot s^!_{\ti G}(\ti \delta-\ti \eps))
%\in A_{d-r}\cM(\si).
%$$

We claim that
\beq\label{Gysin-partial}
 \rho(\si)\lsta ([D]\cdot s^!_{\ti G}(\partial_0\ti R-\partial_\infty \ti R))= 0
 %\rho(\si)\lsta (D\cdot s^!_{\ti G}(\partial_\infty \ti R))=0
 \in A_{d-r} \Msi.
 \eeq
This is true because $s^!_{\ti G} (\partial_0\ti R-\partial_\infty
\ti R)=0\in A_{d-r+1}X$, and $\rho(\si)\lsta$ and $[D]\cdot$
preserve rational equivalence.

% $\ti \rho$ is proper, $ s^!_{\ti G}(\partial_0\ti R)=s^!_{\ti G}(\partial_\infty \ti R)$ in $A_{d-r+1} Y$.
%Since $\rho(\si)\lsta$ and intersecting with $D$ preserve the rational equivalence, we have \eqref{Gysin-partial}.

Applying Lemma \ref{extra-2}, we also have
\beq\label{Gysin-zero}
 \rho(\si)\lsta ([D]\cdot s^!_{\ti G}(\ti B_0))= \rho(\si)\lsta ([D]\cdot s^!_{\ti G}(\ti C_0))=0
 \in A_{d-r} \Msi.
\eeq
Furthermore, because of Lemma \ref{sub} and that $\ti B_1$ and $\ti C_1$ lie over $|D|$,
$$\rho(\si)\lsta ([D]\cdot s^!_{\ti G}(\ti B_1-\ti C_1))=\rho(\si)\lsta s^!_{\ti E|_{|D|}}(\ti B_1-\ti C_1)
=s^!_{E|_{\Msi}}(\ti\rho\lsta \ti B_1-\ti\rho\lsta \ti C_1).
$$

Therefore, using
$\ti B_2-\ti C_2=(\partial_0\ti R-\partial_\infty\ti R)-(\ti B_0-\ti C_0)-(\ti B_1-\ti C_1)$,
 \eqref{Gysin-partial} and \eqref{Gysin-zero}, and $\ti\rho\lsta(\ti B_1)=k\cdot B_1$ and same for $C_1$,
we obtain
$s_{E,\si}^!(\partial_0 R)-s_{E,\si}^!(\partial_\infty R)=0$.
This proves the Lemma.
\end{proof}

\begin{coro}\label{loc-Gysin}
The Basic Construction
defines a homomorphism
$$s_{E,\si}^!: A\lsta E(\si)\to A\lstar\Msi,
$$
which we call the localized Gysin map. Furthermore, if we let $\iota: \Msi\to \cM$
and $\ti \iota: E(\si)\to E$ be the inclusions, then $\iota\lsta\circ s_{E,\si}^!=s_E^!\circ\ti\iota\lsta:
A\lsta E(\si)\to A\lstar M$.
\end{coro}

\begin{proof}
The first part is the combination of Lemma \ref{Gysin-1} and \ref{exist-G}. The second part is
the consequence of Lemma \ref{sub}.
\end{proof}

\begin{exam}
Let $M$ be an $n$-dimensional smooth scheme, $E$ a vector bundle on $M$ and $\si: E\to\cO_M$ a cosection
so that $\si\upmo(0)=p$ is a simple point in $M$. Let $[M]$ be the cycle of the zero section of
$E$ in $Z_n E(\si)$. Then $s_{E,\si}^!([M])=(-1)^n[p]$.
\end{exam}

The proof is straightforward. We blow up $M$ at $p$ to get $\rho: \ti M\to M$
with $D\sub \ti M$ the exceptional
divisor. We let $F=\ker\{ \rho\sta E\to \cO_{\ti M}(-D)\}$, and compute
$s_F^!([\ti M])=c_{n-1}(F)[\ti M] =[D]^{n-1}$. Thus,
$$s_{E,\si}^!([M])=\rho\lsta\bl [-D]\cdot s_F^!([\ti M])\br=\rho\lsta(-[D]^n)=(-1)^n[p].
$$

\section{Localized Gysin maps for bundle stacks}

To construct localized virtual cycles, we need to generalize the localized Gysin map to
bundle stacks over a DM stack $\cM$.

Let $E\bul\in D(\cM)$ be a derived (category) object that is locally quasi-isomorphic to a two-term
complex of locally free sheaves concentrated at $[0,1]$. We let $\bE=h^1/h^0(E\bul)$,
which is a bundle stack isomorphic to $E^1/E^0$ in case $E\bul\cong_{qis}[E^0\to E^1]$
(cf. \cite{BF}).

In case there is a vector bundle $V$ that surjects onto $\bE$, one
defines $s_\bE^!: A\lsta \bE\to A\lsta \cM$ as the composite of
the flat pull-back $A\lsta \bE\to A\lsta V$ and the Gysin map
$s_V^!: A\lsta V\to A\lsta \cM$.

Without such a vector bundle, one can either use the intersection
theory developed by Kresch \cite{Kresch} to define $s_\bE^!$, or
follow the recipe developed in \cite{LT2} by the second named
author. We will  follow the latter approach in this paper.

We suppose there is a surjective homomorphism of sheaves on an open substack $U\sub \cM$
\beq\label{si-comp}
\si: h^1(E\bul)|_U \lra \cO_U.
\eeq
It induces a morphism from the bundle stack $\bE|_U$ to the line bundle $\CC_{U}$.
As before, we let $\Msi=\cM\setminus U$.
We let $\bE(\si)$ be the kernel cone stack in $\bE$,
\beq\label{cone-E}
\bE(\si)=\bE|_\Msi \, \cup \ker\{ \bE|_U\to \CC_U\}\sub \bE
\eeq
endowed with the reduced structure.
(Since $\sigma$ is surjective on $U$, $ \bE|_U\to \CC_U$ is surjective; thus the kernel is well-defined
and is a closed substack of $\bE|_U$.)
%(Note that $\ker\{ \bE|_U\to \CC_U\}=h^1/h^0([E^0|_U\to E^1|_U^{\text{red}}])$, where $E^1|_U^{\text{red}}=
%\ker\{E^1|_U\to h^1(E\bul)|_U \cO_U\}$.

Our goal is to construct a localized Gysin map
\beq\label{Gysin-E}s_{\bE,\si}^!: A\lsta \bE(\si)\lra A\lsta \Msi.
\eeq
It is constructed by finding for each irreducible cycle
$[\bc]\in Z\lsta \bE(\si)$ a proper representative $m_X\upmo[C]\in Z\lsta F$
of $[\bc]$  in a vector bundle $F$ over a variety $X$
with a surjective meromorphic cosection
$(F,\sigma_X)$ and a proper $\rho: X\to\cM$, and then defining
$$s_{\bE,\si}^!([\bc])=m_X\upmo\rho(\si)\lsta(s_{F,\si_X}^!([C]))\in A\lsta \Msi.
$$

We remark here that any homomorphism $V\to h^1(E\bul)$ from a
locally free sheaf $V$ to $h^1(E\bul)$ induces canonically a
bundle-stack morphism $V\to \bE$. Indeed, let $\eta: M\to\cM$ be
an \'etale open so that $\eta\sta E\bul\cong_{qis} [F^0\to F^1]$
with both $F^i$ locally free on $M$. We lift the pull-back
$\eta\sta V\to \eta\sta h^1(E\bul)$ to a homomorphism $\eta\sta
V\to F^1$, which defines a homomorphism of complexes $\eta\sta
V[-1]\to [F^0\to F^1]$. By the functorial construction of
$h^1/h^0$, we obtain a morphism 
$$\eta\sta V = h^1/h^0(\eta\sta
V[-1])\to h^1/h^0(\eta\sta E\bul)=\bE\times_\cM M.
$$
One checks
that any two liftings $\eta\sta V\to F^1$ define homotopic
$\eta\sta V[-1]\to [F^0\to F^1]$, and thus induce identical
(canonical) $\eta\sta V\to\bE\times_\cM M$. Since the so
constructed morphism is canonical, it descends to a morphism
\beq\label{vect} V\lra \bE. \eeq

We now construct representatives of irreducible $[\bc]\in Z\lsta \bE$. For this, we need
the notion of \'etale representatives.
We pick an \'etale morphism $\eta: M\to\cM$ so that $M$ is a scheme, $\eta(M)\cap \pi(\bc)\ne \emptyset$
and that there is a vector bundle $\pi_V: V\to M$ and a surjective homomorphism
$\gamma_M: V\to \eta\sta h^1(E\bul)$. By shrinking $M$ if necessary, we know
$E\bul|_M\cong_{qis}[E^0\to E^1]$ for a complex $[E^0\to E^1]$ of locally free sheaves.
Letting $V=E^1$, the pair $(\eta,V)$ is as desired.
Notice that $\gamma_M$ induces a homomorphism \eqref{vect}.

\begin{defi}\label{eta}
An \'etale representative of an irreducible $[\bc]\sub Z\lsta \bE$ consists of $(\eta,V)$ as before and an irreducible
cycle $m_M\upmo[C_M]\in Z\lsta V$ such that $C_M$ is an irreducible component of $V\times_\bE\bc$,
and $m_M$ is the degree of the \'etale morphism $\eta|_{\pi_V(C_M)}: \pi_V(C_M)\to \pi(\bc)$.
\end{defi}

We next introduce the notion of proper representatives. We let $\rho: X\to \cM$ be
a proper morphism with $X$ a quasi-projective scheme such that $\rho(X)\supset \pi(\bc)$.
Since $X$ is quasi-projective, we can pick a vector bundle $\pi_F: F\to X$ together with a
surjective homomorphism $\gamma_X: F\to \rho\sta h^1(E\bul)$.

We then form $Y=X\times_\cM M$, and let $q_1: Y\to X$ and $q_2: Y\to M$ be the projections.
We pick a vector bundle $\ti F$ on $Y$ and projections $\eta: \ti F\to q_1\sta F$ and $\eta_2: \ti F\to
q_2\sta V$ that make the following square commutative:
$$
\begin{CD}
\ti F @>{\zeta_1}>> q_1\sta F\\
@VV{\zeta_2}V @VV{q_1\sta\gamma_X}V\\
q_2\sta V @>{q_2\sta \gamma_M}>> q_1\sta \rho\sta h^1(E\bul)=q_2\sta \eta\sta h^1(E\bul).
\end{CD}
$$
We let $\ti q_1: q_1\sta F\to F$ and $\ti q_2: q_2\sta V\to V$ be the projections.
Notice that $\ti q_2$ is proper while $\ti q_1$ is \'etale. In the following, for an \'etale
morphism, say $\ti q_1$ and an irreducible cycle $[T]\in Z\lsta q_1\sta F$,
we define $\ti q_{1\ast}([T])=d\upmo[\overline{\ti q_1(T)}]$, where $d=
\deg\{ \ti q_1|_{T}: T\to \ti q_1(T)\}$.

\begin{defi}
Let $(\rho,X,F)$ be as before. We say an irreducible cycle
$m_X\upmo[C_X]\in Z\lsta F$ is a proper representative of
$[\bc]\in Z\lsta \bE$ if for an $(\eta, V)$ in Definition
\ref{eta} and $(\ti F,\zeta_i)$ as constructed, we can find
irreducible $\ti m_X\upmo[\ti C_X]\in Z\lsta q_1\sta F$ and $\ti
m_M\upmo[\ti C_M]\in Z\lsta q_2\sta V$ such that $\zeta_1\sta(\ti
m_X\upmo[\ti C_X])=\zeta_2\sta( \ti m_M\upmo[\ti C_M])$, $ \ti
q_{2\ast}(\ti m_M\upmo[\ti C_M])=m_M\upmo[C_M]$ and $  \ti
q_{1\ast}(\ti m_X\upmo[\ti C_X])=m_X\upmo[C_X]$.
\end{defi}

We remark that the need to introduce \'etale representatives of $[\bc]$ is to make
sense of the ``degree" $m_X$ of the morphism $C_X\to \bc$. Since ``degree" is a birational
property, we use  \'etale representatives to define it.

Using proper representatives of $[\bc]\in Z\lsta\bE$, we can define $s_{\bE,\si}^!$ at the
cycle level.
Let $[\bc]\in Z\lsta \bE(\si)$ be an irreducible cycle. Let $(\rho,X,F)$ with $m_X\upmo[C_X]\in Z\lsta F$
be a proper representative of $[\bc]$. Let %$U_X=\rho\upmo U$ and let
$\si_X: F|_{\rho\upmo U}\to \cO_{\rho\upmo U}$ be the composite of
$F|_{\rho\upmo U}\to \rho\sta\Ob_{\cM/\cS}|_{\rho\upmo U}$
with $\rho\sta\si$. Automatically $[C_X]\in Z\lsta F(\si_X)$. We let $\rho(\si): X(\si_X)\to \cM(\si)$ be
the restricton of $\rho$ to $X(\si_X)=X\setminus\rho\upmo U$.

\begin{defi}\label{Gysin-e}
We define $s_{\bE,\si}^!([\bc])=m_X\upmo \rho(\si)\lsta(s_{F,\si_X}^!([C_X]))$. We extend it to
$s_{\bE,\si}^!: Z\lsta \bE\to A\lsta \Msi$ by linearity.
\end{defi}

\black

\begin{prop} \label{Gysin-prop}
The map $s_{\bE,\si}^!$ in Definition \ref{Gysin-e} is well-defined.
The map $s_{\bE,\si}^!$ preserves the rational equivalence.
\end{prop}

\begin{proof}
We first show that each irreducible $[\bc]\in Z\lsta \bE(\si)$
admits a proper representative. Since $\cM$ is a DM stack, we can
find a quasi-projective variety $X$ and a proper $\rho: X\to \cM$ so
that $\rho: X\to \rho(X)$ is generically finite and $\pi(\bc)\sub
\rho(X)$ is dense. Since $X$ is quasi-projective, there are a vector
bundle $\pi_F:F\to X$ on $X$ and a surjective sheaf homomorphism
$\gamma:F\to\rho\sta  h^1(E\bul)$. This homomorphism induces a
bundle stack morphism $\ti\gamma: F\to \bE\times_\cM X$.

Since $\gamma$ is generically finite, there is an open $O\sub \pi(\bc)$ so that
$\rho|_{\rho\upmo O}: \rho\upmo O\to O$
is flat and \'etale.
We pick an irreducible component of $F\times_{\bE} (\bc\times_\cM O)$,
and let $C_X$ be its closure in $F$. We then let $m_X$ be the degree of
$\rho|_{\pi_F(C_X)}: \pi_F(C_X)\to \pi(\bc)$.
It is routine to check that $m_X\upmo[C_X]$ is a proper representative of $[\bc]$.

Let $\si_X: F|_{\rho\upmo U}\to\cO_{\rho\upmo U}$ be the composite
of $\gamma: F\to \rho\sta h^1(E\bul)$ with $\rho\sta\si$. It is
direct to check that $[C_X]\in Z\lsta F(\si_X)$. Thus
$s_{\bE,\si}^!([\bc])=m_X\upmo \rho(\si)\lsta(
s_{F,\si_X}^!([C_X]))$ is defined.

To check that it is well-defined, we need to show that if $m_{X'}\upmo[C_{X'}]$
with $(\rho', X', F')$ its associated data is another proper representative of $[\bc]$, then
$$m_X\upmo \rho(\si)\lsta( s_{F,\si_X}^!([C_X]))=m_{X'}\upmo \rho'(\si)\lsta( s_{F',\si_{X'}}^!([C_{X'}])).
$$
This can be proved by choosing a third proper representative using
$Y\sub X\times_\cM X'\to \cM$ an irreducible component, choosing
appropriate $\bar F$ that surjects onto the pull-backs of $F$ and
$F'$ and whose projections commute with the projections to the
pull-backs of $ h^1(E\bul)$, and choosing a representative
$m_Y\upmo[C_Y]$ that can be directly compared with $m_X\upmo[C_X]$
and $m_{X'}\upmo[C_{X'}]$. The detail is parallel to the proof of
Lemma \ref{exist-G}, and will be omitted. This defines the
homomorphism $s_{\bE,\si}^!: Z\lsta \bE(\si)\lra A\lsta \Msi$.

Finally, we check that it preserves the rational equivalence. This time the proof is a line by line
repetition of the proof of Lemma \ref{Gysin-1}, incorporating the need to use \'etale representatives
to make sense of degrees of maps. Since the modification is routine, we will omit the details here.
\end{proof}

This proves that the constructions of Definition \ref{Gysin-e}
define a homomorphism \beq\label{Gysin-stac} s_{\bE,\si}^!:
A\lsta\bE(\si)\lra A\lsta \Msi. \eeq

\section{Reducing intrinsic normal cones by cosections}

In this section, we show that a Deligne-Mumford stack equipped with
a perfect obstruction theory and a meromorphic cosection of its
obstruction sheaf has ``restricted'' intrinsic normal cones.
Applying the localized Gysin maps, we obtain the localized virtual cycles.

We let $\pi: \cM\to\cS$ be a DM stack $\cM$ over a smooth Artin stack $\cS$; we assume
$\pi$ is representable.
%\subsection{Statement of the cosection Lemma}
We assume $\cM/\cS$ admits a relative perfect obstruction theory $E\bul\to L\bul_{\cM/\cS}$
formulated in \cite{BF} using cotangent complexes.\footnote{We recall that there
are two versions of perfect obstruction theories. One formulated by Behrend-Fantechi \cite{BF}
using an arrow from a derived object to the relative cotangent complex $E\bul\to \L\bul_{\cM/\cS}$;
the other by Li-Tian using obstruction to deformation assignment in the cohomology of
a derived object.
%In \cite{C-Li} we call the first a perfect obstruction theory and the second a
%cohomological perfect obstruction theory.
By \cite[Theorem 4.5]{BF}, the BF's version of perfect obstruction
theories induces LT's version of perfect obstruction theories.

Conversely,
%In general, it is easier to construct cohomological perfect obstruction theories, but easier to work
%with a perfect obstruction theory. In
it will be shown in \cite{C-Li} that LT's version of perfect
obstruction theory is affine locally equivalent to BF's version.
Furthermore, all available technical tools concerning cycles
constructed from BF's version of perfect obstruction theory work for
LT's version as well.} As part of the definition, locally $E\bul$ is
quasi-isomorphic to two-term complexes of locally free sheaves
concentrated at $[-1,0]$.

We introduce the obstruction sheaf of a relative perfect obstruction
theory. We recall that $\Ob_{\cM/\cS}=h^1((E\bul)\dual)$ is the
relative obstruction sheaf of $E\bul\to L\bul_{\cM/\cS}$. To
introduce its (absolute) obstruction sheaf, we pick a smooth chart
$M/S$ of $\cM/\cS$ by affine schemes $M$ and $S$ such that $S\to\cS$
is smooth, $M\sub \cM\times_\cS S$ is open and $M\to\cM$ is \'etale.
(This is possible since $\cM\to\cS$ is representatble.) We pick an
$S$-embedding $M\to V$ into an affine $V$, smooth over $S$. Since
$M$ is affine, we pick a presentation $E\bul|_M=[E^{-1}\to E^0]$ as
a complex of locally free sheaves so that the perfect obstruction
theory of $\cM/\cS$ lifted to $M/S$ is given by a homomorphism of
complexes of sheaves \beq\label{homo} (\phi\upmo,\phi^0): [E^{-1}\to
E^0]\to  \tau^{\ge -1} L\bul_{M/S}=[I_M/I_M^2\to\Omega_{V/S}|_M],
\eeq where $I_M$ is the ideal sheaf of $M\sub V$. We let $\pi_S:
M\to S$ be the projection.

We denote
$$\Ob_{M/S}\defeq \Ob_{\cM/\cS}\otimes_{\cO_\cM}\cO_M=h^1((E\bul|_M)\dual).
$$
From the distinguished triangle
\[
\pi_S^*L^\bullet_S\lra L_M^\bullet\lra L^\bullet_{M/S}\lra
\pi_S^*L^\bullet_S[1]=[\pi_S^*\Omega_S\to 0],
\]
we have the composition
\[
E^\bullet \lra \tau^{\ge -1} L\bul_{M/S} \lra \pi_S^*\Omega_S[1]
\]
and a distinguished triangle
\[
\hat{E}^\bullet \lra E^\bullet \lra \pi_S^*\Omega_S[1],
\]
which fits into a commutative diagram of distinguished triangles
\beq\label{diag-11}
\begin{CD}
\hat{E}^\bullet @>>> E^\bullet @>>>
\pi_S^*\Omega_S[1] @>{+1}>>\\
@VVV@VVV@|\\
L\bul_M @>>> L\bul_{M/S} @>>> \pi_S^*\Omega_S[1] @>{+1}>>. 
\end{CD}
\eeq
By the
standard 5-lemma, we find that $\hat{E}\bul\to L\bul_M$ is a perfect
obstruction theory of $M$ and that the obstruction sheaf $\Ob_M$ is
the quotient 
\beq\label{quot2} \Ob_{M/S}\lra
\Ob_M=\coker\{\pi_S\sta\Omega_S\dual\lra \Ob_{M/S}\} \eeq of
$\Ob_{M/S}$. 

Since this construction is canonical, the object $\hat E\bul$ is unique up to quasi-isomor-
phism;
the arrow $\pi_S\sta \Omega_S\to \hat E\bul$ (in \eqref{diag-11})
is unique up to homotopy. Thus $\Ob_M=h^1((\hat E\bul)\dual)$ in
\eqref{quot2} is canonically defined, and the first arrow in
\eqref{quot2} is unique. This proves that \eqref{quot2}
descends to a quotient homomorphism of sheaves on $\cM$
\beq\label{ObM} \Ob_{\cM/\cS}\lra \Ob_\cM. \eeq

\begin{defi}
We call $\Ob_\cM$ the (absolute) obstruction sheaf of the obstruction theory
$E\bul\to L\bul_{\cM/\cS}$.
\end{defi}

We denote
$$\bE\defeq h^1/h^0((E\bul)\dual)
$$
and let $\bc_\cM\sub \bE$ be the intrinsic normal cone introduced in
\cite{BF}. A meromorphic cosection of $\Ob_{\cM}$ will reduce the intrinsic
normal cone $[\bc_\cM]$ to a subcone-stack of $\bE$. Let $U\sub\cM$
be an open substack and let \beq \si: \Ob_{\cM}|_U\lra \cO_U \eeq be
a surjective homomorphism. As before, we call $\sigma$ a meromorphic
cosection surjective on $U$; we call $\cM(\sigma)=\cM-U$ the
degeneracy locus of $\si$.

The homomorphism $\sigma$ induces a homomorphism $(E\bul)\dual|_U\to
[0\to\sO_U]$, and thus a surjective cone-stack morphism $\bar\sigma:
\bE|_{U}\lra h^1/h^0([0\to\sO_{U}])= \CC_{U}$. (Here we use
$\CC_\cM$ to denote the trivial line bundle $\Ao\times\cM\to\cM$.)

\begin{defi}
We define $\bE(\sigma)$ to be the union of $\bE\times_\cM
\cM(\sigma)$ with $\ker\{\bar\si:  \bE|_{U}\to \CC_{U}\}$, endowed
with the reduced structure. The closed substack $\bE(\sigma)\sub
\bE$ is called the kernel cone-stack of $\sigma$.
\end{defi}

\begin{prop}\label{cor2.6}
Let the notation be as stated, and let $\si: \Ob_\cM|_U\to\cO_U$ be
surjective. Then the cycle
$[\bc_\cM]\in Z\lsta \bE$
 lies in $Z\lsta \bE(\sigma)$.
\end{prop}

We consider a simple case. Let $M\sub V$ be a closed subscheme of a smooth scheme $V$ defined
by the vanishing $s=0$ of a section $s$ of a vector bundle $E$ on
$V$; let $C_{M/V} $ be the normal cone to $M$ in $V$, embedded in $E|_M$
via the section $s$. We suppose $\bar\si: E\to\cO_V$ is a surjective homomorphism.
Let $I_M$ be the ideal sheaf of $M\sub V$.
The following Lemma was essentially proved in \cite{BGP}.

\begin{lemm}[Cone reduction criterion]\label{j2.11}
Suppose the defining equation $s$ satisfies the following criterion:
for any germ $\varphi:\spec  \kk[\![\xi]\!]\to V$, $\varphi(0)\in X$, the section $\bar\si\circ s\circ\varphi\in \kk[\![\xi]\!]$
satisfies $\bar\si\circ s\circ\varphi\in \xi\cdot \varphi\sta I_M$. Then
the support of the cone $C_{M/V} $ lies entirely in the kernel $F=\ker\{\bar\sigma: E\to\cO_V\} \sub E$.
\end{lemm}

\begin{proof}
%To prove the lemma, we shall view
The cone $C_{M/V}\sub E$ is the specialization
of the section $t^{-1}s\sub E$ as $t\to 0$. More precisely, we
consider the subscheme
$$\Gamma=\{(t^{-1}s(w),t)\in E\times (\Ao\setminus 0)\mid w\in V,\ t\in\Ao\setminus 0\}.
$$
For $t\in \Ao\setminus 0$, the fiber $\Gamma_t$ of $\Gamma$ over $t\in\Ao$ is
the section $t^{-1}s$ of $E$. We let $\bar\Gamma$ be the
closure of $\Gamma$ in $E\times \Ao$. The central
fiber $\bar\Gamma\times_{\Ao}0\sub E$ is the normal cone $C_{M/V} $
embedded in $E|_M$. Clearly, $C_{M/V} $ is of pure dimension $\dim V$.

Now let $N\sub C_{M/V} $ be any irreducible component. % not lying in the zero section of $E$.
Let
$v\in N$ be a general closed point of $N$.
Then we can find a smooth
affine curve $0\in S$ and a morphism $\rho\mh (0,S)\to (v,\bar\Gamma)$ such that $\rho(S\setminus 0)\sub \Gamma$.
We let
$\rho_V: S\to V$ and $ \rho_{\Ao}:  S\to \Ao$
be the composites of $\rho$ with the projections from $E\times \Ao$ to $V$ and to
$\Ao$. Since $\rho(S\setminus 0)\sub \Gamma$, $\rho_{\Ao}$ dominates $\Ao$.

We then choose a
uniformizing parameter $\xi$ of $S$ at $0$ so that
$(\rho_{\Ao})\sta(t)=\xi^n$ for some $n$. Because $\rho(0)=v$,
$\xi^{-n} \cdot s\circ \rho_V:S\setminus 0\to E$ specializes to $v$; hence
$s\circ\rho_V$ has the expression
$$s\circ\rho_V=v\xi^n+O(\xi^{n+1}).
$$
Applying $\bar\si$, we obtain $\bar\si\circ s\circ\rho_V=\bar\si(v)\xi^n+O(\xi^{n+1})$.

Now suppose $N\not\sub E(\sigma)$; in particular $v$
does not lie in the zero section of $E$.
Then $\rho_V\sta I_M=(\xi^n)$. By assumption, $\bar\si\circ s\circ\rho_V\in (\xi^{n+1})=\xi\cdot \rho_V\sta I_M$,
we must have $\bar\si(v)=0$. This proves that $v\in F=\ker\{\bar\si\}$. Since $v$ is general in
$C_{M/V}$, $v\in F$ implies that the support of $C_{M/V}$ lies in $F$. This proves the Lemma.
\end{proof}

Let
$\Ob_M=\coker\{ds: T_V|_M\to E|_M\}$; let
$\pr: E\to \Ob_M$ be the projection.

\begin{coro}\label{j2.1}
Let the notation be as in Lemma \ref{j2.11}. Suppose we have a surjective homomorphism $\sigma\mh\Ob_M\to\cO_M$.
Then the support of the cone $C_{M/V} $ lies entirely in the kernel $F=\ker\{\sigma\circ\pr: E\to\cO_M\} \sub E$.
\end{coro}

\begin{proof}
We verify the cone reduction criterion. Let $\varphi: \spec\kk[\![\xi]\!]\to V$, $\varphi(0)\in M$, be any morphism.
Suppose $\varphi\sta I_M=(\xi^n)$.
Pulling back  the exact sequence
$$\cO_M(T_V)\mapright{ds} \cO_M(E)\lra \Ob_M\lra 0
$$
via the induced morphism
$\bar\varphi=\varphi|_{\varphi\upmo(M)}:\varphi\upmo(M)=(\xi^n=0)
\to M$,
we obtain
\begin{equation}\label{j2.3}
\bar\varphi\sta(T_V) \mapright{\bar\varphi\sta(ds)} \bar\varphi\sta(E)
\mapright{\bar\varphi\sta(\text{pr})} \bar\varphi\sta(\Ob_M)\lra 0,
\end{equation}
where $\bar\varphi\sta(\text{pr})$ is the pullback of the projection $\text{pr}: E|_M\to \Ob_M$.

Let $v\in E|_{\varphi(0)}$ be the element so that
$s\circ\varphi=v\xi^n+O(\xi^{n+1})$. Thus
$\bar\varphi\sta(ds)=d(v\xi^n)= nv\xi^{n-1}d\xi$.

Let $\bar\si=\text{pr}\circ\si$. Because \eqref{j2.3} is exact, we
have the vanishing
$\bar\varphi\sta(\bar\sigma)(\bar\varphi\sta(ds))=0$, which proves
that $\bar\si(v)=0$; hence $\bar\si\circ s\circ
\varphi=O(\xi^{n+1})$. Thus the section $s$ satisfies the cone
reduction criterion (Lemma \ref{j2.11}). This proves the lemma.
\end{proof}

%We let $\jmath: \bE(\sigma)\to\bE$ be the closed immersion.

Assume there is a vector bundle (locally free sheaf of $\sO_\cM$-modules) $F$ that surjects
onto $\Ob_{\cM/\cS}$. %; let $F$ be its associate vector bundle (i.e. $\sO_\cM(F)=\cF$).
This homomorphism induces a flat morphism $F\to
h^1/h^0((E\bul)\dual)$ (cf. \eqref{vect}), which pulls back
$[\bc_\cM]$ to a cycle $[C_\cM]\in Z\lsta F$. We let $\ti\si$ be the
composition \beq\label{sigma-til} \ti\si: F|_U\lra
\Ob_{\cM/\cS}|_U\lra  \Ob_\cM|_U\lra \cO_U, \eeq which is
surjective. Let $F(\ti\sigma)\sub F$ be
$F|_{\Msi}\cup\ker\{\ti\si\}$, endowed with the
reduced structure. %Let $\imath: F(\ti\sigma)\to F$ be the closed immersion.

\begin{coro}\label{cor2.7}
Let the notation be as stated and let $\sigma$ be a surjective homomorphism
$\si:\Ob_\cM|_U\to \cO_U$ over an open $U$. Then the cycle
$[C_\cM]\in Z\lsta F$ lies in
$Z\lsta F(\ti\sigma)$.
\end{coro}

\begin{proof} %[Proof of Corollary \ref{cor2.7}]
Because of the way $F(\ti \si)$ is defined, we only need to show that $[C_\cM\times_\cM U]\in Z\lsta (F(\ti \si)\times_\cM U)$.
By replacing $\cM$ with the open substack $U$, we can assume that $\si$ is regular and surjective on $\cM$.

Since the statement is local, we only need to consider the case where
$M/S\to\cM/\cS$ is as introduced before \eqref{homo}.
Since $M$ is affine, we can pick $E\bul|_M=[E\upmo\to E^0]$ so that $E\bul|_M\to\tau^{\ge -1} L\bul_{M/S}$ is
given by \eqref{homo} and that in addition satisfies $\phi^0: E^0\to \Omega_{V/S}|_M$ is
an isomorphism. Because of the exact sequence $(E^0)\dual\to (E\upmo)\dual\to \Ob_{M/S}\to 0$,
possibly by replacing $E\upmo\to E^0$ with a quasi-isomorphic complex we can assume
$F|_M\to\Ob_{M/S}$ lifts to a surjective $F|_M\to (E\upmo)\dual$, thus defining a homomorphism of
complexes
$$
\gamma: [0\to F]\to [(E^0)\dual\to (E\upmo)\dual].
$$
Two $\gamma$'s coming from two liftings $F\to (E\upmo)\dual$ are homotopy equivalent, and hence the induced morphism of bundle stacks %$\ti\gamma: F\to \bE=h^1/h^0((E\bul)\dual)$
\beq\label{F-cone}\ti\gamma: F\lra \bE=h^1/h^0((E\bul)\dual)
\eeq
is canonically defined.

Let $\ti E\upmo$ be a locally free sheaf on $V$ such that $\ti
E\upmo|_M\cong E\upmo$. This is possible since $V$ is affine. By the
same reason, we can lift $\phi\upmo\in
I_M/I_M^2\otimes_{\cO_M}(E\upmo)\dual$ to an $f\in
I_M\otimes_{\cO_V}(\ti E\upmo)\dual$. Then since $\phi\upmo:
E\upmo\to I_M/I_M^2$ is surjective \cite[Thm. 4.5]{BF}, $M=(f=0)\sub
V$. We let $C_{M/V}$ be the normal cone to $M=(f=0)$ in $V$. The
cone $C_{M/V}$ canonically embeds in $(E\upmo)\dual=(\ti
E\upmo)\dual|_M$ via the defining equation $f$.

On the other hand, because the arrow in $[I_M/I_M^2\to\Omega_{V/S}|_M]$ is via sending
$u\in I_M/I_M^2$ to its relative differential $d_{/S}u\in \Omega_{V/S}|_M$, and because $\phi^0$ is an isomorphism,
after identifying $E^0$ with $\Omega_{V/S}|_M$ using $\phi^0$, the arrow in $[E\upmo\to E^0]$ is
the relative differential $d_{/S}f\in\Omega_{V/S}|_M\otimes_{\cO_M}(\ti E\upmo)\dual$.
Thus $\Ob_{M/S}=\coker\{ d_{/S}f: T_{V/S}|_M\to (E\upmo)\dual\}$.
Following the definition of the obstruction sheaf $\Ob_M$,
$$\Ob_M=\coker\{ df: T_V|_M\lra (E\upmo)\dual\}.
$$

Finally, we apply Corollary \ref{j2.1} to conclude that the support
of the cone $C_{M/V}$ lies in the kernel of the composite
$(E\upmo)\dual\to \Ob_M\to\cO_M$. Since the pull-back of $\bc_\cM$
to $F$ under \eqref{F-cone} is the pull-back of $C_{M/V}$ under the
surjective $F|_M\to (E\upmo)\dual$, and since the support of
$C_{M/V}$ lies in the kernel of $(E\upmo)\dual\to\cO_M$, the support
of $C_\cM|_M$ lies in the kernel $F(\ti\sigma)$. This proves the
Corollary.
\end{proof}

\begin{proof}[Proof of Proposition \ref{cor2.6}]
The proof is a direct application of Corollary \ref{cor2.7}. We pick
an $M/S\to\cM/\cS$ as in the proof of Corollary \ref{cor2.7}; we
only need to consider the case where $M\to\cM$ factors through
$U\sub \cM$. We pick a vector bundle (locally free sheaf) $F$ on $M$
so that $F$ surjects onto $\Ob_{M/S}$. This is possible since $M$ is
affine. Following the proof of Corollary \ref{cor2.7}, $\gamma:
F\to\Ob_{M/S}$ induces a bundle stack homomorphism $\ti\gamma: F\to
\bE\times_\cM M$. Let $\ti\si$ be as in \eqref{sigma-til}. (Note
that $M\to\cM$ factors through $U\sub \cM$.) Since
$\ti\gamma\upmo(\bE(\si))=F(\ti \si)$, the statement of the
Corollary is equivalent to $\ti\gamma\sta[\bc_\cM]\in Z\sta
F(\ti\si)$. But this is what is proved in Corollary \ref{cor2.7}.
This proves the Proposition.
\end{proof}

We have an equivariant version of the Corollary. We suppose $\cM/\cS$
as before has a $\CC\sta$-structure; we suppose its relative obstruction theory
is $\CC\sta$-equivariant.
Let ${\cM^c}$ be the
$\CC\sta$-fixed part
%\footnote{In the sequel, we shall use
%superscript $c$ to indicate the $\CC\sta$-fixed part. Here $c$ stands for {\sl constant}.}
of $\cM$. We suppose there is a surjective sheaf homomorphism
\begin{equation}\label{eq-sigma}
\sigma_c:(\Ob_\cM|_{{\cM^c}})^{\CC\sta}\lra \cO_{{\cM^c}},
\end{equation}
from the $\CC\sta$-fixed part
%$(\Ob_M|_{N})^{\CC\sta}=(\Ob_M\otimes_{\cO_M}\cO_{N})\ucsta$.
to $\cO_{\cM^c}$, and let $F$ be a $\CC\sta$-locally free sheaf of $\cO_\cM$-modules
and $F\to\Ob_\cM$ be a $\CC\sta$-equivariant surjective homomorphism.
We let $\sigma$ %\mh F|_{{\cM^c}}\to\cO_{{\cM^c}}$ be
be the composite
$$\sigma: F|_{{\cM^c}}\lra \Ob_\cM|_{{\cM^c}}\lra (\Ob_\cM|_{{\cM^c}})^{\CC\sta}\mapright{\sigma_c}\cO_{\cM^c},
$$
where the second arrow is the projection onto the invariant part.

\begin{lemm}
Let the notation be as before and let $C_\cM\sub F$ be the cone-cycle that is the pull-back of the intrinsic normal cone
$\bc_\cM$ (cf. Corollary \ref{cor2.7}). Then the support of the restriction $C_\cM|_{{\cM^c}}\sub F|_{{\cM^c}}$
lies in the kernel vector bundle $F|_{\cM^c}(\sigma)=\ker\{F|_{{\cM^c}}\to\cO_{{\cM^c}}\}$.
\end{lemm}

\begin{proof}
%First, by replacing $\cM$ by its \'etale neighborhood, we can assume
%that $M$ is a $\CC\sta$-scheme. Also, since this is a pointwise
%statement, we only need to prove that the cone $C$ over any point
%$p\in {N}\sub M$ lies in the kernel of $\sigma_c(p)\mh E_p\to\CC_p$.
%($E_p=E|_p$.)

We let $M/S\to \cM/\cS$ be as in the proof of Corollary \ref{cor2.7}. Without loss of generality,
we can assume both $S$ and $M$ are $\CC\sta$-schemes and
$M/S\to\cM/\cS$ is a $\CC\sta$-morphism.
(We can avoid this assumption by working with the
formal completion of $M$ at a closed $p\in M\times_\cM \cM^c$; the remainding
arguments go through.)
We take a $\CC\sta$-equivariant $S$-embedding $M\to V$ as before.

We repeat the proof of Corollary \ref{cor2.7}.
Since the obstruction theory is $\CC\sta$-equivariant, we can choose a $\CC\sta$-complex
$E\bul|_M=[E\upmo\to E^0]$ so that the obstruction theory $E\bul|_M\to \tau^{\ge -1} L\bul_{M/S}$ is given by
a $\CC\sta$-equivariant \eqref{homo}.
We extend $E\upmo$ to a $\CC\sta$-equivariant $\ti E\upmo$ on $V$.
Since \eqref{homo} is $\CC\sta$-equivariant, we can choose a
$\CC\sta$-equivariant lift
$f\in I_M\otimes_{\cO_V} (\ti E\upmo)\dual$ of $\phi\upmo$.

We let $V^c$ (resp. $M^c$; resp $E^i_c$) be the $\CC\sta$ fixed part
of $V$ (resp. $M$; resp. $E^i$); let $I_{M^c}$ be the ideal sheaf of
$M^c\sub V^c$. Then the $\CC\sta$-fixed part $f^c\defeq
(f)^{\CC\sta} \in I_{M^c}\otimes_{\cO_{V^c}} (\ti E\upmo_c)\dual$
defines $M^c=(f^{c}=0)\sub  V^c$.

On the other hand, since the $\CC\sta$-invariant part of \eqref{homo} is a perfect obstruction
theory of $M^c$, the cokernel of $df^{c}$, which is a quotient of $(E\upmo_c)\dual$, is the
obstruction sheaf $\Ob_{M^c}$ of $M^c$. One checks that it is identical to the invariant part
$ (\Ob_M|_{M^c})^{\CC\sta}$ (cf. \cite{GraberPand}).

We now look at the normal cone $C_{M^c/V^c}$ (resp. $C_{M/V}$) to
${M^c}$ (resp. $M$) in $V^c$ (resp. $V$); it is a subcone of
$(E\upmo_c)\dual$ (resp. $(E\upmo)\dual$). By the previous
Corollary, the cycle $[C_{M^c/V^c}]$ is a cycle in
$\ker\{(E\upmo_c)\dual\to \cO_{M^c}\}$, where the arrow is the
composite $(E\upmo_c)\dual\to \Ob_{M^c}\mapright{\sigma_c}\cO_{M^c}$
(as in the statement of the Lemma).

We claim that
% the fixed part of
%$C|_p\sub E_p$ is the same as $C_{{\cM^c}}|_p$; namely
\begin{equation}\label{2.6.1}
(C_{M/V})^{\CC\sta}=C_{M^c/V^c}\sub (E\upmo_c)\dual.
\end{equation}
To prove this, %we let $w=(w_1,\cdots,w_m)$ be a multiple variable so
%that $E_p=\spec \kk[w]$; we let $w^c=(w_1,\cdots,w_l)$ be the fixed
%part of $w$, and write
we consider the graph $\Gamma$ of $t^{-1}f$ ($t\in\Ao$),
considered as a subscheme in $(\Ao\setminus 0)\times (\ti E\upmo)\dual$.
By viewing it as a family over $\Ao\setminus 0$, we can take its
$\Ao$-flat closure
$\overline\Gamma\sub \Ao\times(\ti E\upmo)\dual$.
By the definition of normal cone, the central fiber
$\overline\Gamma_0\defeq \overline\Gamma|_{t=0}$  is the cone
$C_{M/V}$ canonically embedded in $(E\upmo)\dual$. Thus $\overline\Gamma_0=C_{M/V}$ and
$(\overline\Gamma_0)^{\CC\sta}=(C_{M/V})^{\CC\sta}$.

On the other hand, notice that $\Gamma$ is $\CC\sta$-invariant with
$\CC\sta$ acting on $\Ao$ trivially; the fixed part $\Gamma^c\defeq \Gamma^{\CC\sta}$
%$\Gamma^c\sub \spec\kk[\![z^c]\!]\times\kk[t,t\upmo]%\times E_p^c$
of $\Gamma$ is
\begin{equation}\label{N-graph}
\Gamma^c=\{\text{the graph of $t^{-1}f^c$}\}\sub
(\Ao\setminus 0)\times (\ti E\upmo_c)\dual.
\end{equation}
We let $\overline{\Gamma^c}$ be the flat closure of $\Gamma^c$ (over $\Ao$).
Thus,
$(\overline{\Gamma^c})|_{t=0}=C_{M^c/V^c}\sub
(E\upmo_c)\dual$.

Therefore, to conclude the Lemma we only need to check that
$(\overline{\Gamma^c})|_{t=0}=(\overline\Gamma_0)^c$.
Because $\CC\sta$ is reductive, the flatness of $\overline\Gamma$ over
$\Ao$ implies that the fixed part $(\overline\Gamma)^{\CC\sta}$ is
also flat over $\Ao$. Then since $(\overline\Gamma)^{\CC\sta}\times_\Ao(\Ao\setminus 0)=\Gamma^c$, and
$\overline{\Gamma^c}$ is $\Ao$-flat, we obtain
$\overline{\Gamma^c}=(\overline\Gamma)^{\CC\sta}$.
This proves $\overline{\Gamma^c}|_{t=0}=(\overline\Gamma_0)^{\CC\sta}$.

Finally, like in the proof of Corollary \ref{cor2.7}, possibly after
replacing $[E\upmo\to E^0]$ by a quasi-isomorphic complex, we can
assume that $F|_M\to \Ob_\cM|_M$ factors through a
$\CC\sta$-equivariant $F|_M\to (E\upmo)\dual\to \Ob_M$. This proves
the Lemma.
\end{proof}

\def\bb{{\mathbf b}}
\section{Localized virtual cycles}

We state and prove the main theorem of this paper.

\begin{theo}\label{main-alg}
Let $\cM/\cS$ be a Deligne-Mumford stack as before endowed with a
relative perfect obstruction theory. Suppose there is a surjective
homomorphism $\si: \Ob_\cM|_U\to\cO_U$ on an open $U\sub\cM$.
Let $\Msi=\cM\setminus U$ be the degeneracy locus of $\si$. Then
$\cM$ admits a localized virtual cycle
$$[\cM]\virt\loc \in A\lsta \cM(\sigma).
$$
It relates to its ordinary virtual cycle by $\iota\lsta[\cM]\virt\loc=[\cM]\virt$,
where $\iota: \cM(\sigma)\to \cM$ is the inclusion.
\end{theo}

\begin{proof}
Let $E\bul\to L_{\cM/\cS}\bul$ be the obstruction theory of $\cM/\cS$. Let
$\bE=h^1/h^0((E\bul)\dual)$ and let $[\bc_\cM]\in Z\lsta \bE$
be the corresponding intrinsic virtual cycle \cite{BF}. The cosection $\si$
defines a surjective bundle stack morphism $\ti\si: \bE|_U\to \CC_U$.
As before, we let $\bE(\si)=\bE|_{\cM(\si)}\cup \ker\{\ti\si\}$,
endowed with the reduced structure; $\bE(\si)$ is a closed substack of $\bE$.

By Proposition \ref{cor2.6}, $[\bc_\cM]$ is a cycle in $Z\lsta \bE(\si)$.
We apply the localized
Gysin map constructed in Proposition \ref{Gysin-prop} to define
$$[\cM]\virt\loc=s_{\bE,\si}^!([\bc_\cM])\in A\lsta \cM(\si).
$$

By the property of localized Gysin map, we have $\iota\lsta[\cM]\virt\loc=[\cM]\virt$.
This proves the Proposition.
\end{proof}

Like the ordinary virtual cycle, the localized virtual cycle is
expected to remain constant in some naturally arisen situations. We
let $\cM/\cS$ as before be a DM stack over a smooth Artin stack
$\cS$, $\cM\to\cS$ representable, with a relative perfect
obstruction theory $E\bul\to L_{\cM/\cS}\bul$. We let $0\in T$ be a
pointed smooth affine curve. We suppose $\cN/\cS$ is a DM stack over
$\cS$, $\cN\to\cS$ representable, together with a morphism $\pi:
\cN\to T$ such that
$$\cM\cong\cN\times_T 0.
$$

We assume there is a perfect relative obstruction theory
$F\bul\to L\bul_{\cN/\cS}$ that is compatible to that of $\cM/\cS$ in that
we have a homomorphism of distinguished triangles in the derived category $D(\cM)$:
\beq\label{diag-BF}
\begin{CD}
 F\bul|_{\cM} @>{g}>> E\bul @>>> \cO_{\cM}[1] @>{+1}>> \\
 @VVV @VVV @VVV\\
 L\bul_{\cN/\cS}|_{\cM} @>{h}>> L\bul_{\cM/\cS} @>>> L\bul_{\cM/\cN} @>{+1}>>.
\end{CD}
\eeq

This way, the (relative) obstruction sheaf $\Ob_{\cN/\cS}=h^1((F\bul)\dual)$
fits into the exact sequence
$\sO_{\cM}\to
\Ob_{\cM/\cS} \to \Ob_{\cN/\cS}|_{\cM}\to 0$. Applying the construction of the (absolute) obstruction
sheaves, we obtain the exact sequence
\begin{equation}\label{2.6}
\cO_{\cM}\lra
\Ob_{\cM} \lra \Ob_{\cN}|_{\cM}\lra 0.
\end{equation}

We suppose there is an open $\cU\sub \cN$ and a surjective homomorphism
$$
\sigma_\cU: \Ob_{\cN}|_\cU \lra \cO_{\cU}.
$$
We let $U=\cU\times_{\cN} \cM$; let $\sigma: \Ob_{\cM}|_{U}\to\sO_{U}$ be the composition of
$\Ob_{\cM}\to \Ob_{\cN}|_{\cM}$ with $\sigma_\cU|_{U}$.
As before, we let $\cN(\sigma_\cU)= \cN\setminus \cU$
and $\cM(\sigma)= \cM\setminus U$
be the degeneracy loci of $\sigma_\cU$. Note that $\cM(\si)=\cN(\si_\cU)\times_T 0$.

We let $\tau: 0\to T$ be the inclusion and let $\tau^!: A\lsta \cN(\si_\cU)\to A\lsta \cM(\si)$ be the
Gysin map that is ``intersecting'' with $0\in T$.

\begin{theo} \label{constancy} Let the notation
be as stated; let $[\cN]\virt\loc\in A\lsta \cN(\si_\cU)$ and $[\cM]\virt\loc\in A\lsta \cM(\si)$ be the
localized virtual cycles. Then
$[\cM]\virt\loc=\tau^!([\cN]\virt\loc)$.
\end{theo}

We will prove the Theorem by applying the rational equivalence inside the deformations of ambient cone-stacks  constructed by
Kim-Kresch-Pantev \cite{KKP}.

We begin with recalling the convention used in \cite{KKP}.
For an object $G\bul$ in the derived category $D(\cM)$ of coherent sheaves on $\cM$,
we denote $p_{\cM}\sta G\bul\otimes p_\Po\sta \sO(-1)$ by $G\bul(-1)$,
where $p_{\cM}$, $p_\Po$ are the two projections of $\cM\times\Po$.
Accordingly, whenever we see a $G\bul\in D(\cM)$ appearing in a sequence involving elements in $D(\cM\times\Po)$,
we understand it as $p_{\cM}\sta G\bul$.
\vsp

For the top line in \eqref{diag-BF}, we mimick the construction of \cite{KKP}. Let
$[x,y]$ be the homogeneous coordinates of $\Po$,  let $\tilde g: F\bul|_{\cM}(-1)\to F\bul|_{\cM}\oplus
E\bul$ be defined by
$(x\cdot 1, y\cdot g)$. The mapping cone $c(\ti g)$ of $\ti g$
is locally quasi-isomorphic to a two-term complex of locally free sheaves,
and fits into the distinguished triangle
\beq\label{ti-g}
{F\bul}|_{\cM}(-1)\mapright{\ti g} {F\bul}|_{\cM}\oplus {E\bul} \lra c(\ti g) \mapright{+1}.
\eeq
Following \cite{KKP}, $h^1/h^0(c(\ti g)\dual)$ is a vector bundle stack over
$\cM\times\Po$, thus flat over $\Po$; its fibers over $a= [1,0]$ and $b=[0,1]\in \PP^1$
are
$$
h^1/h^0(c(\ti g)\dual)\times_\Po a=
h^1/h^0((E\bul)\dual), \ h^1/h^0(c(\ti g)\dual)\times_\Po b=h^1/h^0((F\bul|_{\cM})\dual)\times \Ao.
$$
Here the $\Ao$ in the product is the fiber of
the vector bundle $h^1/h^0([0\to\cO_{\cM}])\cong \CC_\cM$.

We let $\bM^0_{{\cN}/\cS}$ be the deformation of $\cN$ to its normal cone $\bc_{\cN/\cS}$; let
$\bc_{\cM\times \Po /\bM^0_{{\cN}/\cS}}$ be the normal cone to $\cM\times \Po$ in $\bM^0_{{\cN}/\cS}$.
%and let $N_{\cM\times \Po /M^0_{{\cN}/{{\cS}}}}$ be the normal sheaf of $\cM\times \Po$ in $M^0_{{\cN}/{{\cS}}}$.
%By the functoriality of the $h^1/h^0$ construction, we have %$\cD:=\bc_{\cM\times \Po /M^0_{{\cN}/{{\cS}}}}$
We let $\be=h^1/h^0(c(\ti g)^{\vee})$,
which is a bundle stack over $\cM\times\Po$. According to \cite{KKP}, we have a canonical closed immersions
\begin{equation}\label{inclusion}
\bc_{\cN/\cS}\sub h^1/h^0((F\bul)\dual)\and
\cD:=\bc_{\cM\times \Po /\bM^0_{{\cN}/\cS}}\sub \be=h^1/h^0(c(\ti g)^{\vee}).
%\sub  N_{\cM\times \Po /\bM^0_{{\cN}/{{\cS}}}}\cong h^1/h^0(c(\ti h)^{\vee}),
\eeq
%where the isomorphism is proved in \cite{KKP}, and we have the inclusion
%\beq\label{inc-2}h^1/h^0(c(\ti h)^{\vee}) \sub h^1/h^0(c(\ti g)^{\vee}).
%\eeq
%be the intersection of the flat part of $\cD$ with the fiber over
%$c\in \Po$. (Namely, if $B$ over $\Po$ is irreducible, then $c^![B]=0$ if
%$B$ is not flat over $\Po$ and $=[B\times_\Po c]$ if $B$ is flat over $\Po$.)
We will see that $\cD$ is flat over $\Po-b$; for $\bE\defeq h^1/h^0((E\bul)^{\vee})$,
$\be\times_\Po a\cong \bE$, and
\beq\label{a-D}
a^![\cD]=[\cD\times_\Po a]=[\bc_{\cM/{{\cS}}}]\in Z\lsta \bE.
\eeq
(Following
\cite{Fulton}, for $c\in\Po$ we define $c^! [\cD]=[\cD^{\text{fl}}\times_\Ao c]$, where $\cD^{\text{fl}}$ is the
$\Po$-flat part of $\cD$.)
Consequently, (the flat part of) $\cD$ provides a rational
equivalence between $[\bc_{\cM/\cS}]$ and $b^![{\cD}]$.
% let $H_a$ and $H_b$ be its fivers over $a$ and $b$.

%The fiber of (\ref{inclusion}) over $b$ is
%$$\cD_b\subset V|_{\cM}\times_{\cM_{0}}h^1/h^0(F\bul^{\vee}|_{\cM_{0}})=\bH_b.
%$$
%(In this section, for any scheme (stack, cycle)
%over $\Po$, we will use the subscript ``$a$'' (or ``$b$'') to denote its fiber over $a$ (or $b$).)

We let $\cD_b=\cD\times_\Po b$. In \cite{Vistoli} and \cite{KKP},
a canonical rational equivalence $[\cR]\in W_*(\cD_b)$ is constructed
%(denoting $\partial_c\cR:=c^![\cR]$ for $c\in \A^1$)
such that
\beq\label{Vistoli}
\partial_0[\cR]=[b^!\cD] \and \partial_{1}[\cR]=[\bc_{\cM/\bc_{\cN/\cS}}]\in Z\lsta
(\be\times_\Po b). 
\eeq 
Combining, we obtain a pair of rational
equivalences giving the equivalence of
$$[\bc_{\cM/{\cS}}]\sim [\bc_{\cM/\bc_{\cN/\cS}}]\quad\text{via}\quad ([\cD], [\cR]).
$$
This rational equivalence implies $\tau^![\cN]\virt=[\cM]\virt$.

To prove the constancy of the localized
virtual cycles, we need to show that both $[\cD]$ and $[\cR]$ lie in the appropriate
kernel bundle stack.
% and
%introduce the kernel bundle-stack $\be(\sigma)\sub \be$ as follows.
First, $\sigma_\cU:\Ob_\cN|_\cU\to\sO_\cU$  induces $(F\bul)\dual|_\cU\to \cO_\cU[-1]$,
which together with the
$(E\bul)\dual|_U\to \cO_U[-1]$ induced by $\sigma$, defines the two right vertical
arrows $\beta_1$ and $\beta_2$ shown below, and making
the right square a commutative square of homomorphisms of complexes.
Using $\beta_1$ and $\beta_2$, we construct a homomorphism of complexes $\si'$
(shown), which together with $\beta_i$ defines an arrow between distinguished triangles (after restricting to
$U\times\Po$):
\beq\label{beta-1}
\begin{CD}
c(\ti g)\dual|_{U\times\Po}@>>> (F\bul)\dual|_{U}\oplus (E\bul)\dual|_U@>>> (F\bul)\dual|_{U}(1) @>{+1}>>\\
@VV{\si'}V @VV{\beta_1}V @VV{\beta_2}V\\
%p_\Po\sta\cO(-1)[-1]
c(\bx)[-1]@>{(x,y)}>> \sO_{\cM}^{\oplus 2} [-1] @>{\bx=(x,y)[-1]}>> p_\Po\sta\cO_\Po(1)[-1] @>{+1}>>\\
%p_\Po\sta\cO(-1)[-1]@>>> (\sO_{\cM}\oplus\sO_{\cM}) [-1] @>>> \sO_{\cM}(1)[1] @>{+1}>>
\end{CD}
\eeq
%where $\bx=(x,y)[-1]: \sO_{\cM}^{\oplus 2}[-1] \to p_\Po\sta\cO(1)[-1]$.

Let $L$ be the line bundle $\cO_\Po(-1)$.
Then, $c(\bx)\cong_{qis.} p_\Po\sta L$; thus
$$h^1/h^0(c(\bx)[-1])\cong p_\Po\sta L,
$$
which is a line bundle on $\cM\times\Po$.
On the other hand, since both $\beta_1$ and $\beta_2$ are defined and surjective on
$U\times\Po$, $\bar\si$ is surjective on $U\times\Po$.
Thus $\bar\si$ induces a surjective bundle-stack homomorphism (which we
still denote by $\bar\si$)
\beq\label{bar-si}
\bar \si: \be|_{U\times\Po}\lra p_\Po\sta L|_{U\times\Po}. %\cong h^1/h^0(\sO_{\cM}(-1)[-1])|_{U\times\Po},
\eeq

We let $\be(\bar\si)$ be the union of $\be|_{\cM(\si)\times\Po}$
with the kernel $\ker\{ \bar\si\}$, endowed with the reduced
structure; $\be(\bar\si)\sub\be$ is closed. As mentioned, we let
$W\lsta \be_b$ be the group of rational equivalences of
$\be_b=\be\times_{\Po}b$.

\begin{lemm}
The cycles $[\cD]\in Z\lsta \be$ and the rational equivalence $[\cR]\in W\lsta \be_b$ lie in $Z\lsta \be(\bar\sigma)$
and $W\lsta \be(\bar\sigma)_b$, respectively.
\end{lemm}

\begin{proof}
Since $\cR\sub \cD\times_\Po b$, the support of $\cR$ lies in $\cD$. Thus we only need to show that
the support of $\cD$ lies in $\be(\bar\sigma)$. Furthermore, since $\be|_{\cM(\si)\times\Po}
=\be(\bar\si)|_{\cM(\si)\times\Po}$, to prove the Lemma we only need to show that
the support of $\cD\times_{\cM\times\Po}(U\times\Po)$ lies in $\be(\bar\si)|_{U\times\Po}$.
Thus, it suffices to prove the Lemma in case $\cN=\cU$;
namely, $\sigma_\cU$ is regular and surjective everywhere.

In the following, we assume $\sigma_\cN: \Ob_\cN\to\cO_\cN$ is regular and surjective on $\cN$.
Since the statement of the Lemma is local, it
suffices to investigate the situation over $N/S\to \cN/\cS$ for $S$ smooth over
$\cS$, $N\to\cN$ \'etale, and $N\sub \cN\times_\cS S$ is an affine open subscheme.

We then pick a smooth affine scheme $V$ over $S$ and $T$, an embedding $N\to V$ that is
both $S$ and $T$-embeddings. Using $N\to V$, we have a representative
%$\tau^{\geq -1} L_{N/S}\bul=[I_N/I_N^2\to\Omega_{V/S}|_N]$ and
\beq\label{rep-1}
\tau^{\geq -1} L_{N/S}\bul=[I_N/I_N^2\to\Omega_{V/S}|_N]
\and
\tau^{\ge -1}L\bul_{M/S}=[I_M/I_M^2\to\Omega_{V/S}|_M],
\eeq
where $M=N\times_T 0$ be the corresponding chart of $\cM$,
$I_N$ and $I_M$ are the ideal sheaves of $N\sub V$ and $M\sub V$, respectively.

Since $N$ and $V$ are affine, we can assume that there is a vector bundle
(locally free sheaf) $F$ on $N$
(resp. $E$ on $M$) so that
\beq\label{E-F}
F\bul|_N=[F\dual\to \Omega_{V/S}|_N ]
\and E\bul|_M=[E\dual\to \Omega_{V/S}|_{M}],
\eeq
%with $F\upmo=\cF\dual$, $F^0=\Omega_{V/S}|_M$, $E\upmo=\cE\dual$
%and $E^0=\Omega_{V/S}|_{M_{0}}$,
and the diagram
\eqref{diag-BF} restricted to $M$ is represented by the following commuting
homomorphisms of complexes of sheaves
\beq\label{diag-3}
\begin{CD}
[F\dual\to \Omega_{V/S}|_N ]|_{M} @>{g_N}>> [E\dual\to \Omega_{V/S}|_{M}]
@>>> [\sO_M\to 0] \\
@VV{[\phi\upmo,\phi^0]|_{M}}V @VV{[\psi\upmo,\psi^0]}V@VVV\\
 \tau^{\geq -1} L_{N/S}\bul|_{M} @>>> \tau^{\geq -1} L_{M/S}\bul
@>>> [I_{M\sub N}/I_{M\sub N}^2\to 0]\\
\end{CD}
\eeq 
which in addition satisfy $\phi^0=\text{id}:\Omega_{V/S}|_N \to
\Omega_{V/S}|_N$ and $\psi^0=\text{id}$, and the part of the top
line at place $[-1]$ is an exact sequence 
$$0\lra F\dual|_M\lra
E\dual\lra \cO_M\lra 0.
$$
Here for the terms in the second line we use
representatives \eqref{rep-1}.

Since $V$ is affine, we can split this exact sequence
to get
$E\dual \cong F\dual|_{M}\oplus\cO_{M}$.
We then extend $F$ to a vector bundle (locally free sheaf) $\ti F$ on $V$;
because of the isomorphism $E\dual \cong F\dual|_{M}\oplus\cO_{M}$,
$\ti F\oplus\cO_V$ is an extension of $E$.

We now give an explicit description of $\be|_{M\times\Po}$ ($\defeq \be\times_{\cM\times\Po}(M\times\Po)$).
By the canonical construction of $c(\ti g)$, we have a canonical isomorphism
\beq\label{e-iso}
\be|_{M\times\Po}\cong  h^1/h^0((c(\ti g_N)\dual),
\eeq
where $\ti g_N=(x\cdot 1, y\cdot g_N)$ is
$$\ti g_N: [F\dual\to \Omega_{V/S}|_N ]|_{M}(-1)\lra
[F\dual\to \Omega_{V/S}|_N ]|_{M} \oplus [E\dual\to \Omega_{V/S}|_{M}]
$$
(cf. $\ti g$ in \eqref{ti-g}).
Using the isomorphism $E\cong F|_{M}\oplus \cO_{M}$ and that the arrow $F\dual|_{M}\to E\dual$
is the inclusion under the splitting, we see that the mapping cone
\beq\label{ee}
c(\ti g_N)\cong_{qis.}[F\dual|_{M} \to \Omega_{V/S}|_{M}][1] \oplus [\cO_{M}\to 0].
\eeq
Thus, %letting $F_M\defeq F|_{M}$, (and
following the convention that $F|_M(-1)=p_{M}\sta F|_M\otimes p_\Po\sta L$ is
a vector bundle on $M\times\Po$, where $L\cong\cO_\Po(-1)$,
we obtain a tautological flat bundle stack morphism
$$\Phi: F_M(-1)\times\Ao \lra \be_N= h^1/h^0(c(\ti g_N)\dual).
$$
We let $ \cD|_{M\times\Po}\sub \be|_{M\times\Po}$ be the pull-back of $\cD\sub \be$ using
the isomorphism \eqref{e-iso}, and let
$$D=\Phi\sta(\cD|_{M\times\Po})\sub F|_M(-1)\times\Ao.
$$

We next describe the pull-back $\Phi\sta(\be(\bar\sigma))$. ($\bar
\si$ is defined in \eqref{bar-si}.) Like the construction of $\si'$
in \eqref{beta-1}, the surjective homomorphism $\sigma_\cN:
\Ob_\cN\to \cO_\cN$ (surjectivity is assumed at the beginning of the
proof) induces a surjective homomorphism \beq\label{del} \delta: F
\mapright{} h^1((F\bul)\dual) \mapright{}\bar \CC_N. \eeq Cf.
\eqref{E-F}. Pulling back $\delta$ to $N\times\Po$ and twisting it
by $\cO_\Po(-1)$, we obtain the second arrow shown below, which
composed with the projection to the first factor of $F(-1)\oplus
\CC_{N\times\Po}$ defines $\gamma_N$: \beq\label{xiM}
\gamma_N: F(-1)\oplus \CC_{N\times\Po} \mapright{\text{pr}} F(-1) %\mapright{\text{pr}}%h^1((F\bul(-1))\dual)
%=\Ob_\cN\otimes_{\cN}\cO_M(-1)
\mapright{}\bar p_\Po\sta L,
\eeq
where $\bar p_\Po: N\times\Po\to\Po$ is the projection.
%, and
%the last arrow is $\sigma|_N: h^1((F\bul)\dual)\to \CC_N$ pull back to $N\times\Po$ and twisted by $\bar p\sta L$.

We let $K$ be the kernel bundle of $\gamma_N$.
%we let $(F|_M(-1)\times\Ao)(\sigma_{M})=(F(-1)\times\Ao)(\gamma_M)|_{M\times\Po}$.
It is direct to check that
\beq\label{KM}
K|_{M\times\Po}=\Phi\sta(\be(\bar\sigma))\sub F|_M(-1)\times\Ao.
\eeq
This way, $[\cD]\in Z\lsta \be(\bar\si)$ is equivalent to
that the the reduced part $D_{\text{red}}$ (of $D$) lies in $K|_{M\times\Po}$.

To prove $D_{\text{red}}\sub K|_{M\times\Po}$, we give a graph construction of $D$.
We pick a lifting $f_1\in I_N\otimes_{\cO_V}\ti F$
of $\phi\upmo\in I_N/I_N^2\otimes_{\cO_N}F$, and extend $f_1$ to a lifting
$(f_1,f_2)\in I_M\otimes_{\cO_V}(\ti F\oplus \CC_V)$ of $\psi\upmo$.
After shrinking $V$ if necessary, $f_1=0$ defines $N$ and $(f_1,f_2)=0$ defines $M$.
By shrinking $T$ if necessary, we can pick a uniformizing parameter $t\in\Gamma(\sO_T)$ so that
$t=0$ consists of a single point $0\in T$.

We next view $x\upmo$ as a meromorphic section
of $L=\cO_\Po(-1)$ with no zero and only pole at $b=[0,1]$. Then $((tx)\upmo f_1, t\upmo f_2)$
is a section of $F(-1)\oplus \CC_{N\times\Po}$ over $(N\setminus M)\times(\Po-b)$;
we denote by
\beq\label{Ga}
\Gamma\sub  (F(-1)\times\Ao)|_{(N\setminus M)\times(\Po-b)}
\eeq
this section (graph). We let
$\overline\Gamma$ be the closure of $\Gamma$ in $F(-1)\times\Ao$.
According to the construction in \cite{KKP},
$$D=\overline\Gamma\times_{N\times\Po}(M\times\Po)\sub F|_M(-1)\times\Ao.
$$

We now prove $D_{\text{red}}\sub K|_{M\times\Po}$. First, let
$$\Gamma_1\sub (F\oplus \CC_N)|_{N\setminus M}
$$
be the graph of $(t\upmo f_1,t\upmo f_2)$
(recall $t\in\Gamma(\cO_T)$ and $M=N\cap (t=0)$); let $\overline{\Gamma_1}$ be its closure in
$F\oplus \CC_{N}$; then $\overline{\Gamma_1}\times_{N}M$ is the normal cone $C_{M/V}$
embedded in $F|_M\oplus\CC_M=F|_M\times\Ao$ by the defining equation $(f_1,f_2)$ of $M$.

We next form the commutative diagram
$$\begin{CD}
(F(-1)\oplus \CC_{N\times \Po})|_{N\times(\Po-b)}
@>{(x\cdot 1, 1)}>>  \bar p_N\sta(F\oplus \CC_N)|_{N\times(\Po-b)}\\
@VV{\gamma_N}V @VV{(\bar p_N\sta \delta,0)}V\\
\bar p_\Po\sta L|_{N\times(\Po-b)} @>{x}>> \CC_{N\times\Po}|_{N\times(\Po-b)}.
\end{CD}
$$
Note that the two horizontal arrows are isomorphisms of vector bundles. By our construction,
$\Gamma$ is the preimage of $\bar p_N\sta \Gamma_1$ under the top horizontal arrow.
Therefore, $\overline\Gamma|_{N\times(\Po-b)}$ is the preimage of
$\bar p_N\sta \overline{\Gamma_1}|_{N\times(\Po-b)}$, and thus
$$D|_{M\times (\Po-b)}\sub (p_M\sta F|_M\times\Ao)|_{M\times(\Po-b)}
$$
is the preimage of $p_M\upmo C_{M/V}\sub p_M\sta(F|_M\oplus \CC_M)$.
This proves that $\cD$ is flat over $\Po-b$ and also the identity \eqref{a-D}.

Furthermore, by Lemma \ref{j2.1}, the support of
$C_{M/V}\sub E|_M\cong F|_M\oplus \CC_M$ lies in the kernel of $(\delta,0): F|_M\oplus \CC_M\to \CC_M$.
Thus the support of $p_M\upmo C_{M/V}$ lies in the kernel of $(\bar p_N\sta\delta,0)$, and
the support of $D|_{M\times(\Po-b)}$ lies in the kernel of $\gamma_N$.

It remains to show that every irreducible component of $D$ that lies over $M\times_\Po b$
lies in $K|_{M\times\Po}$ (cf. \eqref{KM}). Let $A\sub D$ be an irreducible
component lying over $M\times_\Po b$,
and let %. Since $D\sub  \overline\Gamma$,
%$A$ is properly contained in an irreducible component $N\sub \overline\Gamma$ such that
%$N\cap \Gamma\ne\emptyset$.
$v\in A$ be a general closed point. Since $\overline\Gamma$ is irreducible, we can
find a smooth curve $0\in S$ and a morphism $\ti\rho: (0,S)\to (v, \overline\Gamma)$ so that
$\ti\rho(S\setminus 0)\in \Gamma$.
We let $(\rho_1,\rho_2): S\to M\times\Po$ be $\ti\rho$ composed with the projection
$\overline\Gamma\to M\times\Po$. By shrinking $S$ if necessary, we can assume
$\rho_2(S)\sub \Po-a$.

We let $x_S\defeq x\circ\rho_2$ (resp. $t_S\defeq t\circ \rho_1$) be the composition of
$x\in\Gamma(\cO_{\Po-a})$ with $\rho_2$ (resp. $t\in\Gamma(\cO_T)$ with the composite $S\mapright{\rho_1} M\to T$).
Since $v=\lim_{s\to 0} \ti\rho(s)$, by the definition of $\Gamma$, we have
$$v=(v_1,v_2)=\bl\lim_{s\to 0}\, (x_S\cdot t_S)\upmo \cdot f_1\circ\rho, \
\lim_{s\to 0}\, (t_S)\upmo\cdot f_2\circ \rho\br.
$$
Let $u=\rho_1(0)\in M$.
Repeating the proof of Lemma \ref{j2.11} and Corollary \ref{j2.1}, the surjective homomorphism
$\si_\cN: \Ob_\cN\to\cO_\cN$ forces
$v_1\in \ker\{ \delta(u): F|_u\to \CC\}$ (cf. \eqref{del}). Finally, by the explicit form of $\gamma_N$
in \eqref{xiM},
$\ker\{\delta(u): F|_u\to \CC\}\times\Ao
=K|_{u\times b}$. Therefore,
$v_1$ lies in this kernel, which proves
$A\sub (F|_M(-1)\times\Ao)(\xi_M)$. This proves the Lemma.
\end{proof}

\begin{proof}[Proof of Proposition \ref{constancy}]
This follows from that the localized Gysin map $s_{\be,\bar\si}^!$ preserves rational equivalence.
\end{proof}

\section{Application: Localized GW-invariants}

We let $X$ be a smooth quasi-projective variety equipped
with a holomorphic two-form $\theta\in \Gamma(\Omega_X^2)$. This form induces a cosection
of the obstruction sheaf of 
$\mgn(X,\beta)$, the moduli space of $n$-pointed genus $g$ stable
morphisms to $X$ of class $\beta$. This cosection defines a localized virtual class of
$\mgn(X,\beta)$, thus a localized GW-invariants of $(X,\theta)$.

%In this section, we shall also prove the deformation invariance of
%such invariants under some technical condition; we shall derive a
%formula of the localized invariants in a special case, sufficient to
%recover all GW-invariants of surfaces without descendants, first
%proved by Lee-Parker \cite{Lee-Parker}.

%\subsection{Cosection of the obstruction sheaf}

We begin with the construction of the cosection of the obstruction
sheaf of $\mgn(X,\beta)$. For simplicity, we will fix the data $g$,
$n$, $X$ and $\beta$ for the moment and abbreviate $\mgn(X,\beta)$
to $\cM$. We let $f\mh \cC\to X$ and $\pi\mh \cC\to \cM$ be the
universal family of $\cM$. If we denote by $\cS$ the Artin stack of
genus $g$ connected nodal curves, $\cM$ is a representable DM stack
over $\cS$; and the relative obstruction sheaf of the standard
relative obstruction theory of $\cM/\cS$ (cf. \cite{BF}) is
$\Ob_{\cM/\cS}=R^1\pi\lsta f\sta T_X$.

We now show that a holomorphic two-form $\theta$ defines a cosection of $\Ob_{\cM/\cS}$.
Indeed, by viewing it as an anti-symmetric homomorphism
\begin{equation}\label{4.1}
\hat{\theta}:T_X\lra \Omega_X,\quad (\hat\theta(v),v)=0,
\end{equation}
it defines the first arrow in the following sequence of homomorphisms
\begin{equation} \label{yh2.2}
R^1\pi_*f^*T_X\lra R^1\pi_*f^*\Omega_X\lra
R^1\pi_*\Omega_{\cC/\cM}\lra R^1\pi_*\omega_{\cC/\cM},
\end{equation}
where the second is induced by
$f\sta\Omega_X\to\Omega_{\cC/\cM}$, and the last by the tautological $\Omega_{\cC/\cM}\to \omega_{\cC/\cM}$.
 Because
$R^1\pi_*\omega_{\cC/\cM}\cong \cO_\cM$, the composite of this sequence provides
\begin{equation}\label{yh2.7} \sigma_\theta^{\text{rel}}:R^1\pi_*f^*T_X=\Ob_{\cM/\cS}\lra
 \cO_\cM.
\end{equation}

The obstruction sheaf of $\cM$ is the cokernel of $p\sta T_\cS\to \Ob_{\cM/\cS}$, where $p:\cM\to\cS$ is
the projection. Using the universal family $f$ and that $R^1\pi_*f^*T_X=\ext^1_\pi(f\sta\Omega_X,\cO_\cC)$,
we have the exact sequence
\beq\label{ob-M}
\cE xt^1_{\pi}(\Omega_{\cC/\cM}, \cO_{\cC})\lra \ext^1_\pi(f\sta\Omega_X,\cO_\cC)\lra \Ob_\cM\lra 0,
\eeq
where the first arrow is induced by $f\sta\Omega_X\to\Omega_{\cC/\cM}$.

\begin{lemm}\label{lem-yh2.3}
The composition
%\begeq \label{yh2.9}
$$ \cE
xt^1_\pi(\Omega_{\cC/\cM},\cO_{\cC}) \lra \ext^1_\pi(f\sta\Omega_X,\cO_\cC) \mapright{\sigma_\theta^{\text{rel}}} \cO_\cM
$$%\endeq
is trivial.
\end{lemm}

\begin{proof}
Applying the definition of $\sigma_\theta^{\text{rel}}$, one sees that the stated composition is the composition of the following
sequence
%\beq\label{comp}
$$\ext^1_\pi(\Omega_{\cC/\cM},\cO_{\cC}) \lra \ext^1_\pi(f\sta\Omega_X,\cO_\cC)\lra
\ext^1_\pi(f\sta T_X,\cO_{\cC}) \lra
\ext^1_\pi(\omega_{\cC/\cM}^\vee,\cO_\cC)
$$
that is induced by the sequence
$$%\omega_{\cC/S}^\vee \ar[r]&
\Theta:\omega_{\cC/\cM}^\vee\lra f^*T_X\mapright{f\sta\hat{\theta}}
f^*\Omega_X \lra\Omega_{\cC/\cM},
$$
where the first arrow is the dual of the composite $f\sta\Omega_X\to\Omega_{\cC/\cM}\to\omega_{\cC/\cM}$.

We now prove that the composite $\Theta=0$. First, let $\cC_{\text{reg}}$ be the smooth loci of the fibers of $\cC/\cM$.
Since $\Omega_{\cC/\cM}|_{\cC_{\text{reg}}}=\omega_{\cC/\cM}|_{\cC_{\text{reg}}}$, and
since  $\hat\theta$ (in \eqref{4.1}) is anti-symmetric,
$\Theta|_{\cC_{\text{reg}}}$ is anti-symmetric. Therefore, because $\omega_{\cC/\cM}$ is invertible,
$\Theta|_{\cC_{\text{reg}}}=0$.

Now let $q\in\cC\setminus \cC_{\text{reg}}$; let $\xi=\pi(q)\in \cM$.
We pick an affine scheme $M$ and an \'etale $M\to\cM$ whose image contains $\xi$; let $\bar\xi\in M$ be a lift of $\xi$.
We let $C$ be an affine open
$C\sub \cC\times_\cM M$ such that $(q,\xi)\in \cC\times_\cM M$ lifts to a $\bar q\in C$.
We let $g: C\to X$ be the restriction of $f$ to $C$, and let $\bar\xi\in M$ be the image of $\bar q$ under $C\to M$.
Since both $C$ and $M$ are affine,
we can find a closed immersion $M\sub \ti M$ into a smooth scheme $\ti M$ and extend $C/M$ to a family of nodal curves
$\tilde C/\tilde M$ so that the node $\bar q\in C_{\bar \xi}$ is smoothed in the family $\tilde C/\ti M$, and
that the morphism $g: C\to X$ extends to $\tilde g: \ti C\to X$.

For the family $\tilde g: \ti C\to X$, we form the similarly defined
homomorphism
\begin{equation}\label{4.6}
\tilde\Theta: \omega\dual_{\tilde C/\tilde M}\lra \Omega_{\tilde
C/\tilde M}.
\end{equation}
Like $\Theta$, $\ti\Theta$ vanishes away from the singularities of the fibers of $\ti C/\ti M$.
On the other hand, since $\ti M$ is smooth and the node $\bar q\in \ti C_{\ti \xi}$ is smoothed
in the family $\ti C/\ti M$, $\Omega_{\ti C/\ti M}$ has no torsions near $\bar q$. Therefore,
that $\ti\Theta$ vanishes away from the singular points of the fibers of $\tilde C/\tilde M$
implies that $\ti\Theta$ vanishes near $\tilde q$.

Finally, since $\Theta|_{C_{\bar \xi}}=\ti\Theta|_{C_{\bar \xi}}$, we conclude that $\Theta$
vanishes at $q\in \cC$. Since $q\in \cC$ is an arbitrary node, this shows that $\Theta=0$.
This proves that the composite in the statement of the Lemma vanishes.
\end{proof}

\begin{coro}
The homomorphism $\si_\theta^{\text{rel}}$ lifts to a homomorphism
$\sigma_\theta: \Ob_\cM\to\cO_\cM$.
\end{coro}

\begin{proof}
This follows from Lemma \ref{lem-yh2.3} and the exact sequence \eqref{ob-M}.
\end{proof}

%This cosection is surjective away from those stable morphisms that
%are $\theta$-null.

\begin{defi} A stable map $u\mh C\to X$ is called \emph{$\theta$-null}
if the composite
$$u\sta(\hat\theta)\circ du: T_{C\reg}\lra u\sta T_X|_{C\reg}\lra
 u\sta\Omega_X|_{C\reg}
$$
is trivial over the regular locus $C\reg$ of $C$.
\end{defi}

\begin{prop}
Any holomorphic two-form $\theta\in H^0(\Omega^2_X)$ on a smooth
quasi-projective variety $X$ induces a homomorphism $\sigma_\theta\mh
\Ob_\cM\lra \cO_{\cM}$ of the obstruction sheaf $\Ob_\cM$ of the
moduli of stable morphisms $\cM=\mgn(X,\beta)$. The homomorphism
$\sigma$ is surjective away from the set of $\theta$-null stable maps in $\cM$.
\end{prop}

\begin{proof}
Only the last part needs to be proved.
Since $\sigma_\theta$ is the lift of $\sigma_\theta^{\text{rel}}$, it is surjective at $[u\mh C\to X]\in\cM$
if and only if $\sigma_\theta^{\text{rel}}$ is surjective at $[u]$. Since the latter at $[u]$ is the composition
$$
H^1(C,u^*T_X)\mapright{u\sta \hat{\theta}} H^1(C,u^*\Omega_X)\lra
H^1(C,\Omega_C)\lra H^1(C,\omega_C)=\CC,
$$
whose Serre dual is
$$\CC=H^0(C,\cO_C)\lra H^0(f^*T_X\otimes \omega_C)\lra
H^0(f^*\Omega_X\otimes \omega_C).
$$
Because $\cO_C$ is generated by global sections, the composite of
the above sequence is trivial if and only if the composite
$$
T_{C}\otimes \omega_C|_{C\reg}\lra f^*T_X\otimes \omega_C|_{C\reg}
\lra f^*\Omega_X\otimes \omega_C|_{C\reg}
$$
is trivial. But this is equivalent to $u$ being $\theta$-null. This
proves the Proposition.
\end{proof}

Using the cosection $\sigma_\theta$, we can
localize the virtual cycle of $\cM$. Let $\cM(\si_\theta)\sub\cM$ be
the collection of $\theta$-null stable maps. Clearly, $\cM(\si_\theta)
\sub\cM$ is closed. Because $\si_\theta: \Ob_\cM\to\cO_\cM$ is surjective away
from $\cM(\si_\theta)$, applying Theorem \ref{main-alg}, we obtain the
localized virtual cycle
$$[\cM]\virt\loc\in A\lsta \cM(\si_\theta).
$$

In case $\cM(\si_\theta)$ is proper, we define the localized GW-invariants
as follows. We let
$\ev: \cM\to X^n$
be the evaluation morphism, let $\gamma_1,\cdots,\gamma_n\in
H\sta(X)$, let $\alpha_1,\cdots,\alpha_n\in \ZZ^{\geq 0}$, and let
$\psi_i$ be the first Chern class of the relative cotangent line
bundle of the domain curves at the $i$-th marked point.
The localized GW-invariant of $X$ with descendants is defined to be
$$\langle\tau_{\alpha_1}(\gamma_1)\cdots\tau_{\alpha_n}(\gamma_n)
\rangle^{X,\mathrm{loc}}_{g,\beta} =\int_{[\cM]\virt\loc}
\ev\sta(\gamma_1\times\cdots\times\gamma_n)\cdot
\psi_1^{\alpha_1}\cdots\psi_n^{\alpha_n}
$$

In case $X$ is proper, then $\cM(\si_\theta)$ is automatically proper.

\begin{lemm} If $X$ is proper, the localized GW-invariant coincides
with the ordinary GW-invariant of $X$.
\end{lemm}

\begin{proof}
This follows from the last statement in Theorem \ref{main-alg}.
\end{proof}

A Corollary of this is the following generalization of the vanishing results of J. Lee and
T. Parker \cite{Lee-Parker, Lee} for compact algebraic surfaces.

\begin{coro}[First vanishing]\label{vanLP}
Let $X$ be a smooth projective variety endowed with a holomorphic
two-form $\theta$. The virtual cycle of the moduli of stable morphisms
$\mgn(X,\beta)$ is trivial unless the class $\beta$ can be represented
by a $\theta$-null stable morphism.
\end{coro}

\begin{proof}
Let $\cM=\mgn(X,\beta)$. Suppose $\beta$ cannot be represented by a $\theta$-null
stable morphism, i.e. $\cM(\si_\theta)=\emptyset$.
Then obviously $[\cM]\virt\loc=0$. Since $[\cM]\virt$ is the image of $[\cM]\virt\loc$ in $A\lsta \cM$,
$[\cM]\virt=0$.
\end{proof}

A generalization of the second vanishing of Lee-Parker is as follows.
Since $\hat \theta$ is a homomorphism from $T_X\to \Omega_X$,
$\det\hat\theta$ can be viewed as a section of the line bundle $K_X^{\otimes 2}$.
Let $D=(\det\hat\theta=0)$. Every $[u,C]\in \cM(\si_\theta)$
has $u(C)\sub D$. Then an easy argument shows that

\begin{coro}[Second vanishing] \label{van-second}Let $(X,\theta)$ be as in Corollary \ref{vanLP}.
For $\beta\ne 0$ and $\gamma_i\in H^*(X)$, the GW-invariant $\langle
\prod \tau_{\alpha_i}(\gamma_i)\rangle^X_{g,\beta}$ vanishes if one
of the classes $\gamma_i$ is Poincar\'e dual to a cycle disjoint
from $D$.
\end{coro}

Further applications of this cycle localization technique will appear in the sequel of this paper.

\section{Other Applications}\label{sec-last}

We conclude our paper with comments on other possible applications;
some has been worked out and some are under development.

One application is the study of extremal GW-invariants of the Hilbert schemes of
points of surfaces. After Beauville, we know that every homomorphic two-form
$\theta$ of an algebraic surface $X$ induces a holomorphic two-form of the
Hilbert scheme $X^{[k]}$ of $k$-points of $X$. Thus the second vanishing
applied to this case gives us a lot a vanishing of the GW-invariants of $X^{[k]}$.
For general $X$, we can pick a meromorphic two-form $\theta$. If $\beta\in H_2(X^{[k]},\ZZ)$
is an extremal class (i.e. a class in the kernel of $H_2(X^{[k]},\ZZ)\to H_2(X^{(k)},\ZZ)$),
then the meromorphic $\theta$ induces a meromorphic two-form on $\mgn(X^{[k]},\beta)$.
This form can be used to study the GW-invariants of $X^{[k]}$. This has been
exploited by W-P. Li and the second named author in \cite{LL}.

Another application is in the study of Donaldson-Thomas invariants
of Calabi-Yau threefold. In \cite{KL3}, the authors showed that the
modified Kirwan blow-up of the moduli of semistable sheaves has an
obstruction theory whose obstruction sheaf has a regular surjective
cosection. Using this cosection, we can localize the virtual cycle
and prove a wall crossing formula using master space construction
and $\CC\sta$-localization. In comparison with Joyce's proof of the
wall crossing formula, our proof does not use the Chern-Simons
functional, and thus can possibly apply to a wider classes of moduli
of derived objects over Calabi-Yau three-folds.

The other potential application is the study of the GW-invariants of a
three-fold $X$ that is a
$\Po$-bundle over a surface $S$ equipped with a holomorphic two-form $\theta$.
The two-form $\theta$ on $S$ pulls back to a two-form of $X$. Since the
geometry of this two-form is explicit, one hopes that one can essentially
reduce the study of the GW-invaraints of $X$ to a ruled surface over the
canonical divisor $D\sub S$. It will be interesting to see some part of this
carried out in near future.

\begin{appendix}
\section{Analytic analogue of localized Gysin maps}

In the case of schemes over $\CC$, we can
construct an analytic analogue of the localized Gysin map by picking a smooth ($C^\infty$)
section of bundle and intersecting with the cycle representative. This was
the original approach adopted in \cite{Kiem-Li}. Due to its potential application, we include
it here. We will also prove the equivalence of the two constructions.

For simplicity, we will work out the case of schemes. The case of orbifolds (DM stacks)
is similar using multiple sections.
Let $\pi: E\to M$ be a rank $r$ complex vector bundle over a reduced scheme $M$, $U\sub M$ an open subset
and $\sigma: E|_U\to \cO_U$ be a surjective homomorphism. We denote $M(\si)=M\setminus U$
and denote $E(\sigma)$ the kernel cone of $\si$, as defined before. We assume $M(\si)$ is compact.

We first pick a splitting of $\sigma$ over $U$.
Because $\si$ is surjective, possibly by picking a hermitian metric on $E$
% and writing $E|_{M-D(\si)}$ as the direct sum of $E(\si)|_{M-D(\si)}$
% with its orthogonal complement $\underline{\CC}_{M-D(\si)}$,
we can find a smooth section $\check{\si}\in C^\infty\bl
E|_{U}\br$ so that $\si\circ\check{\si}=1$.
Next, we pick an analytic neighborhood $V$ of
$M(\sigma)\sub M$ such that $V$ has compact (analytic) closure in $M$ and
the homomorphism
$$\imath\lsta: H\lsta(M(\si),\ZZ)\mapright{\cong} H\lsta(V,\ZZ)
$$
induced by the inclusion $\imath: M(\si)\to V$ is an isomorphism. Because $M$ is algebraic
and $M(\si)$ is compact, such a neighborhood
$V$ always exists (\cite{Mather}).
We then extend $\check \sigma|_{M-V}$ to a
smooth section $\check\si_{\text{ex}} \in C^\infty\bl E\br$.
and pick a
smooth function $\rho\mh M\to \RR^{>0}$ so that
$\xi=\rho\cdot\check\sigma_{ex}\in C^\infty\bl E\br$
is a small perturbation of the zero section of $E$.

Now let $B\subset E(\si)$ be a complex $d$-dimensional closed subvariety. By fixing a
stratification of $B$ and of $M$ by complex subvarieties, we can
choose the extension $\check{\si}_{ex}$ and the function $\rho$ so
that the section $\xi$ intersects $B$ transversely. As a
consequence, the intersection $B\cap \xi$, which is of pure
dimension, has no real codimension 1 strata. Henceforth, it defines
a closed oriented Borel-Moore cycle in $E$.

But on the other
hand, since $\sigma\circ\xi |_{M-V} =\rho\in C^\infty\bl
M-V\br$, $\xi$ is disjoint from $B$ over $M-V$. Thus
$T\cap\xi\sub E|_V$. Adding that $V$ has compact (analytic) closure in $M$, $\pi(B\cap \xi)$
defines a closed chain in $V$, thus defines a homology class
$[\pi(B\cap \xi)]\in H_{2d-2r}(V,\ZZ)$.
Finally, applying the inverse of  $\imath\lsta$, we define
$$s_{E,\si} ^{!,\text{an}}([B])=\imath\lsta\upmo([\pi(B\cap \xi)])\in H_{2d-2r}(M(\si),\ZZ).
$$

Applying the standard transversality argument, one easily shows that
this class is independent of the choice of $V$ and the section
$\xi$; thus it only depends on the cycle $B$ we begin with.

\begin{defiprop}\label{2.5}
The map
$s_{E,\si} ^{!,\text{an}}: Z_d E(\sigma)\to H_{2d-2r}(M(\sigma),\ZZ)$
so defined descends to a homomorphism
$$s_{E,\si} ^{!,\text{an}}: A_d E(\sigma)\lra H_{2d-2r}(M(\sigma),\ZZ),
$$
which we call the analytic localized Gysin map. Furthermore, via the cycle-to-homology
homomorphism $cl: A_k M(\si)\to H_{2k}(M(\si),\ZZ)$, the two version of
localized Gysin maps coincide, i.e. $cl\circ s_{E,\si}^!=s_{E,\si}^{!,\text{an}}$.
\end{defiprop}

\begin{proof}
The proof that $s_{E,\si} ^{!,\text{an}}$ preserves rational equivalence is standard and
will be omitted. We now prove the comparison by using the notation of Section 2.
If $B\subset E|_{M(\si)}$, the result follows directly from \cite[Lemma 19.2 (d)]{Fulton}. So we suppose $B\cap G\ne \emptyset$. We choose a $\si$-regularizing morphism $\rho:X\to M$ and a closed integral $\ti B\subset \ti G$ such that $\ti\rho_*[\ti B]=k\cdot [B]$ where $\ti G$ is the kernel of $\ti \si: \ti E\to \cO_X(D)$ with $\ti E=\rho^*E$ and $\ti\rho:\ti E\to E$ is induced from $\rho$.

Observe that the intersections $\xi\cap B$ for defining $s_{E,\si}^{!,\text{an}}[B]$
and $D\cdot s^!_{\ti G}[\ti B]$ for defining $s_{E,\si}^{!}[B]$ take place in the tubular
neighborhood $V$ of $M(\si)$ in the sense that if we replace $M$ by $V$ and $X$ by
the inverse image of $V$ and so on, we get the same classes. So for the purpose of
the comparison result, we may assume that $M=V$ is a tubular neighborhood of $M(\si)$.

Consider the fiber square of zero sections
\[
\xymatrix{ X\ar@{^(->}[r]^{\ti \imath} \ar[d]_\rho & \tilde{E} \ar[d]^{\ti \rho} \\
M\ar@{^(->}[r]^\imath &E .
}
\]
By \cite[Theorem 6.2]{Fulton}, $$\rho_*(s^!_{\ti E}[\ti B])=s_E^!\ti\rho_*[\ti B]=k\cdot s_E^![B].$$
By the excess intersection formula \cite[Theorem 6.3]{Fulton},
\[
\rho_*(D\cdot s_{\ti G}^![\ti B])=\rho_*(s^!_{\ti E}[\ti B])=k\cdot s_E^![B].
\]
Applying the cycle-to-homology homomorphism $cl$, we obtain
\[
cl(s_{E,\si}^![B])=k^{-1}\cdot cl(\rho_*(D\cdot s_{\ti G}^![\ti B]))=cl(s_E^![B]).
\]
By \cite[Lemma 19.2 (d)]{Fulton} again, this equals $u_\imath\cap cl([B])$, where $u_\imath\in H^{2r}(E,E-M)$ is the orientation class of $E$. By standard arguments, the cap product of the orientation class is the same as intersecting $B$ with a transversal perturbation $\xi$ of the zero section. This proves the Proposition.
\end{proof}

\end{appendix}

\bibliographystyle{amsplain}

\end{document}